\documentclass[a4paper,twoside,12pt]{article}      
\usepackage{amsmath,amssymb,amsfonts} 
\usepackage{amsthm}
\usepackage{graphics}                 
\usepackage{color}                    
\usepackage{hyperref}                 
\usepackage{graphicx,eucal}
\usepackage{latexsym,epsfig}

\usepackage{graphicx}                 

\usepackage{amssymb}

\newtheorem{thm}{Theorem}
\newtheorem{cor}[thm]{Corollary}
\newtheorem{lem}[thm]{Lemma}
\newtheorem{prop}[thm]{Proposition}
\newtheorem{rem}[thm]{Remark}

\newcommand{\F}{\mathcal{F}}

\newcommand{\G}{\mathcal{G}}

\newcommand{\D}{\mathcal{D}}
\newcommand{\divR}{\textup{div}^{\textup{h}}}
\newcommand{\DR}{\D^{\textup{h}}}

\newcommand{\fin}{\mbox{}\hfill $\square$ \\[0.2cm]}

\newcommand{\radial}{\mathcal{R}}
\newcommand{\simfor}{\sim_0}

\newcommand{\dd}{\textup{d}}
\newcommand{\corr}[1]{{#1}}

\newenvironment{pf}{\par\noindent\textbf{Proof. }}
{\fin}
\newenvironment{pfof}[1]{\par\noindent
\textbf{Proof of #1. }}
{\fin}

\title{Classification and moduli spaces of dicritical singularities}
\author{Gabriel {\sc Calsamiglia}, Yohann {\sc Genzmer}\footnote{The first author is supported by CNPq/Faperj; the second author supported by the ANR under the project ANR-13-JS01-0002-01. Both authors are supported by the Math-Amsud project DFH.}}

\begin{document}
\maketitle

\begin{abstract}
  In this paper we give complete analytic invariants for germs of holomorphic foliations in $(\mathbb{C}^2,0)$ that become regular after a single blow-up. Some of them describe the holonomy pseudogroup of the germ and are called transverse invariants. The other invariants lie in finite dimensional complex vector space.  Such singularities admit separatrices tangent to any direction at the origin. When enough separatrices coincide with their tangent directions  (a condition that can always be attained if the mutiplicity of the germ at the origin is at most four) we are able to describe and realize all the analytical invariants geometrically and provide analytic normal forms. As a consequence we  prove that any two such germs sharing the same transverse invariants are conjugated by a very particular type of birational transformations. We also provide the first explicit examples of universal equisingular unfoldings of foliations that cannot be produced by unfolding functions. With these at hand we are able to explicitely parametrize families of analytically distinct foliations that share the same transverse invariants. 
\end{abstract}

\tableofcontents

\section{Introduction and main statements}

In this paper, we deal with analytic invariants, normal forms and unfoldings of germs of holomorphic foliations in $(\mathbb{C}^2,0)$. Two such foliations are said to be analytically equivalent if there exists a germ of biholomorphism  of $(\mathbb{C}^2,0)$ sending leaves of one to leaves of the other. There are two known objects that have analytic properties and contain some of the analytical information of the foliation. On the one hand, there is the the analytic structure of the {\bf holonomy pseudogroup} formed by holonomy maps associated to paths in the leaves of the foliation and transverse sections at the endpoints. These maps are holomorphic as soon as the foliation is. On the other, we know that there is a non-empty set of separatrices through $0$. These are leaves $L$ such that $L\cup 0$ is a germ of analytic curve at $0$. The analytic class of the {\bf set of separatrices} also contains analytical information of the foliation. In the 70's, R. Thom asked whether the analytic class of these two objects is enough to determine the analytical class of the foliation. When the union of all separatrices is an analytic curve, there are instances where the answer is positive. For example, for generic homogeneous foliations, namely foliations whose separatrix set is a homogeneous curve and whose singularities after one blow-up are of hyperbolic type, the analytical class of the curve and of the projective holonomy representation determine the analytical type of the singularity (see \cite{Ma} or \cite{Genzmer}). In particular, foliations defined by holomorphic vector fields whose multiplicity at the origin is less than five and with generic first homogeneous term, fall in the previous case.

A different approach to the problem was taken by J.-F. Mattei in the 80's (see \cite{M}). In the spririt of Kodaira and Spencer's theory of deformations of complex structures on manifolds, he proved that any germ $\mathcal{F}$ of holomorphic foliation of $(\mathbb{C}^2,0)$ can be unfolded to a {\em codimension one} germ of foliation $\widetilde{\mathcal{F}}$ on $(\mathbb{C}^{2+\mathcal{M}(\mathcal{F})},0)$ in such a way that any other unfolding of $\mathcal{F}$ that preserves the singularity type of $\mathcal{F}$ is equivalent to one obtained from $\widetilde{\mathcal{F}}$ in a unique manner by pull back.  He calculated the dimension $\mathcal{M}(\mathcal{F})$ of the base space of this {\em universal equisingular unfolding} of $\mathcal{F}$ and concluded that it is always finite and almost always positive. By construction, the deformation of $\mathcal{F}$ obtained by considering the foliations $\{\mathcal{F}_c: c\in(\mathbb{C}^{\mathcal{M}(\mathcal{F})},0)\}$ of $(\mathbb{C}^2,0)$ obtained by intersecting $\widetilde{\mathcal{F}}$ with the fibres of the projection $(\mathbb{C}^{2+\mathcal{M}(\mathcal{F})},0)\rightarrow (\mathbb{C}^{\mathcal{M}(\mathcal{F})},0)$ has the same singularity type and  holonomy pseudogroup for all parameters (up to equivalence), but are not analytically equivalent. In some sense, the moduli of Mattei tell us in how many ways we can locally change the analytical class of the foliation without changing the holonomy pseudo-group. In the case of homogeneous foliations, these moduli are simply the relative position of the points in the tangent cone of the set of separatricies: actually, they characterize the analytical class of the separatrix set. In general, we are not able to interpret geometrically the other moduli of Mattei. Indeed, the construction of the latter is not explicit, producing the foliations by foliated surgery. Up to now, the only explicit examples of such non-trivial unfoldings were obtained by unfolding germs of functions. In fact, the first examples of foliations, where the answer to Thom's question is definitively negative, can be constructed by unfolding functions (see \cite{GP}). For example, the foliation defined by the germ at zero of the function $$f(x,y,z)=(1+z)xy(x+y)(x-y)(x+2y+y^2)$$
defines a non-trivial unfolding of the foliation $\mathcal{F}=\{f(x,y,0)=const\}$ having the same separatrix set for all parameters $z\in(\mathbb{C},0)$.

Yet another approach has been taken in recent years by Ortiz-Rosales-Voronin (see \cite{orv, orv1, orv2}). Their stragtegy is, on the one hand to find unique formal normal forms (up to formal transformations tangent to the identity) for certain families of foliations, and on the other to prove that formal analytical rigidity takes place in the generic cases. Hence the coefficients of the formal normal form turn out to be analytical invariants. This infinite number of parameters is then split into two subsets: one of them is infinite and contains the information on the holonomy pseudo-group and the other is finite and contains the rest of parameters. The number of parameters that is not associated to the holonomy pseudo-group coincides with the number of Mattei's parameters although it is not clear how the formal deformations obtained in the formal normal forms correspond to unfoldings in the sense of Mattei. Again, it is not clear what these parameters mean geometrically.

In this paper, we will give a precise description of invariants, their geometric interpretation, analytic normal forms, unfoldings and moduli spaces for a particular class of foliations admitting an infinite number of separatrices: homogeneous dicritical foliations. Most of our arguments will be geometric. In a second instance, we will apply formal methods to try to generalize the claims for foliations in a wider class: the class of germs that are regular after a single blow-up.

\subsection{Homogeneous dicritical foliations.}

A germ of foliation $\mathcal{F}$ of $(\mathbb{C}^2,0)$ is said to be homogeneous dicritical if it becomes regular after a single blow-up and there exists a foliation $\mathcal{G}$ defined by a holomorphic vector field with radial linear part $x\partial_x+y\partial_y$  such that $\text{Tang}(\mathcal{F},\mathcal{G})$ is {\em invariant}. In this case, we say that $\mathcal{F}$ is homogeneous with respect to $\mathcal{G}$. In other words, up to a change of coordinates (one that linearizes $\mathcal{G}$) there exists a subset of separatrices that is a union of straight lines and supports the tangency set of the foliation with the radial foliation. The set of all germs of foliations that are regular after a single blow up will be denoted by $\mathcal{D}$ and for each $n$ we denote by $\mathcal{D}(n)$ the set of foliations in $\mathcal{D}$ of multiplicity $n+1$. In particular $\mathcal{D}(0)$ corresponds to foliations with radial linear part, and will be called radial foliations. By Poincar\' e's linearization theorem, every radial foliation is holomorphically linearizable. The subset of $\mathcal{D}$ of homogeneous dicritical foliations will be denoted by $\DR$ and its subset $\DR(n)=\mathcal{D}(n)\cap \DR$ is formed by those having multiplicity $n+1$ at the origin. The following are examples of elements in $\DR(n)$ for $n\geq 1$:
\begin{enumerate}
  \item Consider homogenous polynomials $R(x,y)$ and $Q(x,y)$ of degrees $n$ and $n+2$ respectively such that $R$ and $xQ$ are coprime. The foliation defined by a holomorphic one form in a neighbourhood of $0\in(\mathbb{C}^2,0)$, $$\omega(x,y)=(R(x,y)+\cdots) (xdy-ydx)+Q(x,y)dx$$ lies in $\DR(n)$.
\item \label{ex:contracting separatrices} Consider a smooth rational curve $C$ embedded in a complex surface $S$ with  self-intersection $(n+2)$. Suppose that $S$ is bifoliated by a pair of regular {\em transverse} holomorphic foliations $\mathcal{F}$ and $\mathcal{G}$ and that $\mathcal{G}$ is transverse to $C$ at all points. Then the blow up of $S$ at $n+3$ points in  $C\setminus\text{Tang}(\mathcal{F},C)$ produces two foliations $\widetilde{\mathcal{F}}$ and $\widetilde{\mathcal{G}}$ around a $(-1)$-curve that can thus be contracted to a couple of germs of foliations. By construction, the germ of foliation associated to the initial $\mathcal{F}$ is homogeneous with respect to the radial foliation associated  to $\mathcal{G}$, and thus belongs to $\DR(n)$.
\end{enumerate}
Beyond these examples, we can state the following result.
\begin{thm}\label{t:small multiplicity}
 Every germ of plane holomorphic foliation of multiplicity at most four that is regular after a single blow up is homogeneous dicritical. For higher multiplicities, the equivalent statement is false.
\end{thm}
 By Camacho-Sad's index Theorem \cite{CS}, for any foliation $\mathcal{F}\in\mathcal{D}$, the pull-back foliation by the blow-up map, denoted by $\widetilde{\mathcal{F}}$ must be generically transverse to the exceptional divisor $E$. Thus, the foliations in $\mathcal{D}$ are \emph{dicritical}: they have an infinite number of invariant curves. In fact, every leaf is a separatrix. We can define the \emph{tangency divisor} $T(\mathcal{F})\in\text{div}(E)$ of tangency  between $\widetilde{\mathcal{F}}$ and $E$. It is an effective divisor whose degree is $n$ if and only if $\mathcal{F}$ belongs to $\mathcal{D}(n)$. Klughertz showed in \cite{klu} that two foliations $\mathcal{F},\mathcal{F}'\in\mathcal{D}$ are topologically equivalent if and only if the divisors $T(\mathcal{F})$ and $T(\mathcal{F}')$ are. Thus any partition of $n=n_1+\ldots+n_k$ defines a topological class in $\mathcal{D}(n)$ formed by foliations $\mathcal{F}$ having $T(\mathcal{F})=n_1p_1+\ldots n_kp_k$ for some pairwise distinct points $p_i$ in $E$.
\begin{figure}
\center{\includegraphics[scale=.7]{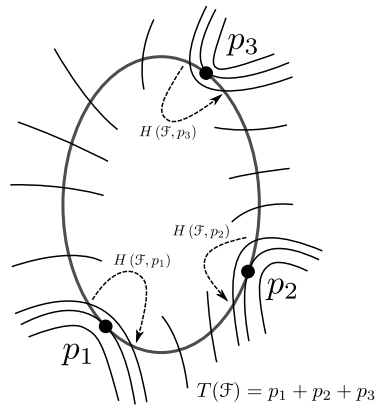}}
\caption{A foliation in $\mathcal{D}$ with three points in its divisor of tangency $T(\mathcal{F})$.}
\end{figure}

The \emph{holonomy pseudogroup} of such a foliation is quite simple: for any $p\in |T(\mathcal{F})|$, we can consider a local primitive holomorphic first integral $f$ of $\widetilde{\mathcal{F}}$ around $p$. The levels of the restriction $f_{|E}$ describe sets of points that belong to the same leaf. Since $f_{|E}$ is a holomorphic germ in one complex variable, it can be written as a power $\psi^{r+1}$, where $r$ is the order of tangency between $\widetilde{\mathcal{F}}$ and $E$ at $p$, and $\psi$ is a holomorphic germ of diffeomorphism. If $\theta_p$ denotes the rotation of angle $2\pi/(r+1)$, the germ $H(\mathcal{F},p):=\psi^{-1}\circ\theta_p\circ\psi$ and any of its powers realizes a holonomy map around each point sufficiently close to $p$. Up to reducing the neighbourhood where $\mathcal{F}$ is defined, two different $H(\mathcal{F},p)$'s cannot be composed so the only holonomy maps are powers of elements of $H(\mathcal{F},p)'s$. Thus $$H(\mathcal{F})=\bigsqcup_{p\in T(\mathcal{F})}H(\mathcal{F},p)$$ is the generating set of the holonomy pseudogroup of the germ $\mathcal{F}$ and we will call it the {\em holonomy} of $\mathcal{F}$.  If $T(\mathcal{F})$ is equal to $n_1p_1+\ldots+n_kp_k$, the holonomy pseudogroup of $\mathcal{F}$ is a disjoint union of finite cyclic groups of orders $n_1+1, \ldots, n_k+1$ and it determines the topological class of $\mathcal{F}$.

If $\mathcal{F}'=\phi (\mathcal{F})$ for a holomorphic equivalence $\phi$, then by construction $$H(\mathcal{F}')=\phi_{|E}\circ H(\mathcal{F})\circ \phi_{|E}^{-1}$$ where $\phi_{|E}$ is the \emph{global} holomorphic automorphism of $E\cong\mathbb{P}^1$ induced by $\phi$. Thus the class of $H(\mathcal{F})$ modulo global automorphisms of $E$ is an analytical invariant of $\mathcal{F}$ that we will denote by $H[\mathcal{F}]$ and call the {\em holonomy class} of $[\mathcal{F}]$. This invariant does not characterize the analytical class of the foliation in general, but in the homogeneous case it characterizes the class modulo a very particular family of birational maps.
 \begin{thm}\label{t:construction}
 Any homogeneous dicritical foliation $\mathcal{F}$ is equivalent to one obtained by blowing up a foliation $\mathcal{F}_S$ on a surface $S$ as in example \ref{ex:contracting separatrices}. The analytical class of $\mathcal{F}_S$ in a neighbourhood of the rational curve is uniquely determined by $H[\mathcal{F}]$.
 \end{thm}
If we consider that \emph{local birational tranformations} stands for a map composed of sucessive changes of coordinates, contractions of compactified regular separatices and blow-ups of regular points as illustrated in Figure \ref{contra}, we deduce directly
 \begin{cor}\label{c:birational tranfs}
  Any pair of dicritical homogeneous foliations sharing the same holonomy class are locally birationally conjugated.
 \end{cor}

\begin{figure}
\center{\includegraphics[scale=0.7]{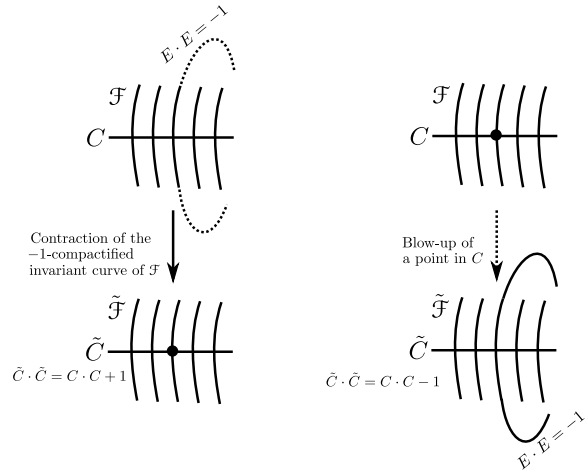}}
\caption{Local birational transformations in $\D$}\label{contra}
\end{figure}

To identify the analytical invariants other than the holonomy $H[\mathcal{F}]$, we proceed to describe the possible ways of obtaining the foliation as in Theorem \ref{t:construction}. Recall that for a homogeneous curve the holomorphic class of the tangent cone determines the analytical class of the curve. For a homogeneous dicritical foliation $\mathcal{F}\in\DR (n)$ we define $$\divR(\mathcal{F})=\{\text{Tang}(\mathcal{F},\mathcal{G})_{|E}\in\text{div}(E): \mathcal{G}\in\mathcal{D}(0), \mathcal{F} \text{ is homogeneous with respect to } \mathcal{G}\}.$$
Since $\mathcal{F}$ might be homogeneous with respect to different radial foliations, it is a non-empty subset of the set of divisors in $E$ of degree $n+3$. Each of them contains the information on the analytical class of a divisor supported on a {\em homogeneous} set of  separatrices of $\mathcal{F}$.

\begin{thm}[Invariants in $\DR$]\label{t:invariantsinDR}
  Let $\mathcal{F}_1\in\DR(n)$ and $\mathcal{F}_2$ be a germ of holomorphic foliation in $(\mathbb{C}^2,0)$. Then there exists $\varphi\in \textup{Aut}(\mathbb{C}^2,0)$ such that $\varphi_*(\mathcal{F}_1)=\mathcal{F}_2$ if and only if $\mathcal{F}_2\in\mathcal{D}^R(n)$, and there exists $\phi=\varphi_{|E}\in \textup{Aut}(E)$ such that
  \begin{itemize}
    \item $H(\mathcal{F}_2)=\phi\circ H(\mathcal{F}_1)\circ\phi^{-1}$
    \item for some $D\in \divR (\mathcal{F}_1)$ (and a posteriori for all), $\phi_*(D)\in \divR (\mathcal{F}_2)$.
  \end{itemize}
\end{thm}
Remark that in $\DR$, all the invariants can be read on the exceptional divisor $E$. Coming back to Thom's question, in the case of a homogeneous dicritical foliation, the equivalence class of the pair $(H(\mathcal{F}),S_D)$ where $S_D$ is the divisor of leaves over an element $D\in\divR(\mathcal{F})$ classifies the analytical class of the foliation. It is worth remarking that the homogeneous separatrix set $|S_D|$ is not enough to classify. We need the divisorial information.

Next we will use the previous results to construct normal forms in a geometric manner. Recall that a Weierstrass polynomial is a monic polynomial in $\mathbb{C}\{x\}[y]$.

\begin{thm}[Normal Forms in $\DR$]\label{t:normalformsinDR}
 Let $\mathcal{F}\in\DR(n)$ and suppose $D\in \divR (\mathcal{F})$. Then there exist coordinates $(x,y)$ of $(\mathbb{C}^2,0)$ where $\mathcal{F}$ is represented by a form \begin{equation}\label{eq:normalformDR}W(x,y)(x\dd y-y\dd x)+Q(x,y)\dd x\end{equation} satisfying that $W$ is a Weierstrass polynomial in $y$ of degree and order $n$, $Q$ is a homogeneous polynomial of degree $n+2$ such that $xQ=0$ represents $D$, and $Q(1,y)$ is monic. With these conditions the form is unique up to a choice of affine coordinate in $E$ and local biholomorphisms tangent to a homothety.
 \end{thm}
The Weierstrass polynomial in the normal form encodes the information on the holonomy of the associated foliation. Indeed, as we will see in section \ref{s:holonomy}, given a degree $n+2$ homogenous polynomial $Q$  as in the statement of Theorem \ref{t:normalformsinDR}, the map that assigns to any admissible Weierstrass polynomial the generator of the holonomy of the normal form obtained with the pair $(W,Q)$, is a bijection onto the space of holomorphic germs of rotations around the points on $E$ determined by  $j^n(W)$.

Since for small multiplicity, all foliations in $\mathcal{D}$ are homogenoeus we get the following
 \begin{cor}\label{c:normalforms4}
   For any germ of foliation $\mathcal{F}$ of multiplicity at most four that is regular after a blow up we can provide a finite list of analytic normal forms which yields a complete system of invariants for the analytical class of $\mathcal{F}$.
 \end{cor}
 The non-uniqueness (in general) of the normal forms for $n\leq 3$ comes from the lack of uniqueness of the radial foliation with respect to which $\mathcal{F}$ is homogeneous. However some special choice of topological class in $\divR(\mathcal{F})$ allows to conclude that we can find a finite list. For $\mathcal{F}\in\mathcal{D}(2)$ it is the class of the form $5p$ for some point $p\in E$; for $\mathcal{F}\in\mathcal{D}(3)$ it is the class of the form $4p_1+p_2+p_3$ for some $p_i\in E$.

 \subsection{Unfoldings of homogeneous dicritical foliations}

 In the light of Theorem \ref{t:construction}, we have a natural way of unfolding a dicritical homogeneous germ $\mathcal{F}\in\DR(n)$ in the space $\DR(n)$ with base space of dimension $n+3$. Indeed, we can suppose $\mathcal{F}$ is obtained from $\mathcal{F}_S$ of Theorem \ref{t:construction} by a blow up on some divisor $D=p_1+\ldots+p_{n+3}$ on the exceptional curve. When two points $p_i$ and $p_j$ coincide, we interpret that we blow up twice at the same point. By considering an affine coordinate $z\in\mathbb{C}$ in the rational curve where $p_i$ corresponds to $z_i\in\mathbb{C}$, and defining $p_i(c)=z_i+c_i$ for $c=(c_1,\ldots,c_{n+3})\in(\mathbb{C}^{n+3},0)$, the foliations $\{\mathcal{F}_c: c\in(\mathbb{C}^{n+3},0)\}$ defined by blowing up $\mathcal{F}_S$ on the divisor $D(c)=p_1(c)+\ldots+p_{n+3}(c)$ form an equisingular unfolding of $\mathcal{F}$. The knowledge of $\divR (\mathcal{F})$ and Theorem \ref{t:invariantsinDR} allow to decide which of the directions in this unfolding are non-trivial. For instance, in the cases $n=2,3$, this procedure applied to the points of the special choices of topological class in $\divR(\mathcal{F})$ of Corollary \ref{c:normalforms4} generically produce universal non-trivial equisingular unfoldings.

 In fact, by this procedure we are able to produce the first explicit examples of non-trivial unfoldings of a foliation without any special integrability properties. Up to now, the only class of foliations for which we were able to exhibit non-trivial unfoldings were foliations admitting a Louvillian first integral, that is, a holomorphic multi-valued first integral. For instance, as we will see in Section \ref{s:holonomy}, if $r\in\mathbb{C}[t]$ has degree $n\geq 2$ the pull-back of  \begin{equation}
\nu=x^{n+2}\dd \left(r\left(\frac{y}{x}\right)+x\right)
\end{equation}  via the rational map $$\Lambda(x,y,(c_{ij}))=(x,y)\cdot\left(1+\sum_{j=1, i<j}^{n-1} c_{ij}x^{i-j}y^{j}\right)$$ produces the universal equisingular unfolding of the foliation associated to $\nu$.

 The most difficult property to obtain in order to construct an unfolding is the integrability property $\Omega \wedge \dd\Omega=0$ which is not automatically satisfied for a one form $\Omega$ in at least three variables. In the theorem below, we describe a non-trivial unfolding for any foliation in the class $\DR$.

\begin{thm}[Unfoldings without first integral]\label{t:unfoldings}
Consider an element of $\DR(n)$ defined by the one-form
\[\omega=\underbrace{\left(R(x,y)+\sum_{i=0}^{n-1}a_i(x)y^ix^{n-i}\right)}_{W(x,y)}(x\dd y-y\dd x)+Q(x,y)\dd x\]  as in the statement of Theorem \ref{t:normalformsinDR} where $R$ is homogeneous of degree $n$ and $a_i\in \mathbb{C}\{x\}$ with $a_i(0)=0$. Then the one-form in $(\mathbb{C}^2,0)\times\mathbb{C}^{n-1}$ defined by $$\Omega =\left( R(x,y)+\sum_{i=0}^{n-1}a_i(x+\left<c,y\right>)y^ix^{n-i}\right)(x\dd y-y\dd x)+Q(x,y)\dd(x+\left<c,y\right>)$$ where $\left<c,y\right>=\sum_{i=1}^{n-1}c_iy^i $, defines a non-trivial equisingular unfolding of any foliation in $\mathcal{D}$ associated to $\left.\Omega\right|_{c=c_0}$ for some $c_0\in\mathbb{C}^{n-1}$. In particular, $\Omega$ satisfies
\[\Omega \wedge \dd \Omega=0.\]

For generic $a_i$'s, $\Omega$ does not admit any type of reasonably analytic first integral (meromorphic, Liouvillian, multivalued and holomorphic on the complement of a countable union of analytic sets,etc.)
\end{thm}
Remark that we get an unfolding whose parameter space is a Zariski open set of $\mathbb{C}^{n-1}$, not just a germ of unfolding. Remark also that for $n\geq 5$, even if the initial germ $\omega$ defines a homogeneous dicritical foliation, the unfolding will not stay in the space of homogeneous dicritical foliations.  We know from the results of Mattei that the base space of the universal unfolding of an element in $\mathcal{D}(n)$ has dimension $n(n-1)/2$, so the obtained unfolding is only part of it if $n>2$. In the case of $n=2$, this number coincides with $n-1$, the dimension of parameters we got in Theorem \ref{t:unfoldings}. All these germs of unfoldings can actually be considered together and compactify the parameter space to get
\begin{cor}\label{c:universal unfolding}
  Given germs $(a,b)\in \mathbb{C}\{x\}$, $b(0)=0$ and $a(0)\in\{0,1\}$,  the foliation in $(\mathbb{C}^2,0)\times \mathbb{P}^1$ defined by $$\Omega(x,y,c)=\left(y^2+a(x)y(x+cy)+b(x)\right)(x\dd y- y\dd x)+(x+cy)^4\dd x$$ describes the universal equisingular unfolding $\{\mathcal{F}_c =\{\Omega_c=0\}:c\in\mathbb{C}\}$ of all elements in $\mathcal{D}(2)/\sim$ having the same holonomy class as $\mathcal{F}_0$. By varying $(a,b)$, we cover all the analytic classes in $\mathcal{D}(2)$.
\end{cor}
From this corollary, we deduce that we have nice parametrizations of $\mathcal{D}(2)/\sim$ to describe the equivalence relation {\em having the same holonomy class}. Indeed, by considering the parameters $(a,b,c)\in\mathbb{C}\{x\}^2\times\mathbb{P}^1$, the fibres of the projection $(a,b,c)\mapsto (a,b)$ parametrize equivalence classes. These restrictions of the parametrization are locally injective in general but not globally injective. As a consequence, up to changing coordinates, any two distinct germs in $\mathcal{D}(2)$ sharing the same holonomy class can be joined by a deformation underlying a non-trivial equisingular unfolding.

\subsection{Formal normal forms and $E$-equivalence in $\mathcal{D}$}
As was seen, the homogeneity hypothesis was a great help in the previous results. Having a radial foliation that is well related to a given foliation allowed in particular to find good coordinates. In general, we do not have such an object but some formal results based on the ideas coming from the homogeneus case can be used to determine general formal normal forms.

We start by generalizing the results in \cite{C} to all of $\mathcal{D}$ with a slight change in the type of equivalences. Two foliations are said to be (formally) $E$-equivalent and denoted $\sim_E$ (respectively $\hat{\sim}_E$),  if there exists a (formal) biholomorphism that is tangent to a homothety $\lambda\text{Id}$ sending one to the other. Geometrically, the property means that the lift of the transformation to the first blow-up on source and target fixes every point of the exceptional divisor $E$.

A monic polynomial in $\mathbb{C}[[x]][y]$ will be called a formal Weierstrass polynomial. If it converges, it is a Weierstrass polynomial.
\medskip

Let $n\geq 1$. For any  $W$ {\bf formal} Weierstrass polynomial in $y$ of degree and order $n$ and any family of complex numbers $(c_{ij})\in \mathbb{C}^{\frac{n(n-1)}{2}}$, we consider the formal foliation given by

\begin{equation}\label{eq:formalnormalform}
\F_{W,(c_{ij})}:=  W(x,y)(x\dd y-y\dd x)+\left(x^{n-1}+\sum_{\tiny \begin{array}{c}0\leq i\leq n-2 \\0\leq j\leq n-1 \\ i+j\geq n-1\end{array}}c_{ij}x^iy^j\right)x^3\dd x.
\end{equation}

\begin{thm}[Formal normal forms in $\mathcal{D}(n)$]\label{t:formalnormalforms}
Consider three distinct points $p_0,p_1,p_\infty \in E$ and $n\geq 1$. For any $\mathcal{F}\in\mathcal{D}(n)$ such that $p_0\notin |T(\F)|$, there exist a formal conjugacy $\Phi\in \widehat{ \textup{Diff}}(\mathbb{C}^2,0)$ and a \textbf{unique} pair $\{W,(c_{ij})\}$ such that
\begin{itemize}
\item $\textup{D}\Phi(p_0)=0,\ \textup{D}\Phi(p_1)=1$ and $\textup{D}\Phi(p_\infty)=\infty$
\item
$\Phi_*\F=\F_{W,(c_{ij})}$
\end{itemize}
\end{thm}
In particular, two formal normal forms $\F_{W,(c_{ij})}$ and $\F_{W',(c'_{ij})}$ are formally $E$-equivalent if and only if they are equal. On the other hand we know from \cite{C} that formal/analytic rigidity takes place in $\mathcal{D}$ so formal invariants are in fact analytic invariants. Thus, we define for three distinct points $p_0,p_1,p_\infty \in E$ and any class of $E$-equivalence $[\mathcal{F}]$ such that $p_0\notin |T(\F)|$, \[{\bf c}[\F](p_0,p_1,p_\infty):=(c_{ij})\in \mathbb{C}^{\frac{n(n-1)}{2}}.\]

These new invariants complement the holonomy invariants of the foliation up to $E$-equivalence as is stated in next

\begin{thm}[$\sim_E$-invariants]\label{t:invariants}
  $\mathcal{F}_1\sim_E\mathcal{F}_2\in \mathcal{D}$ if and only if $\mathcal{F}_1\in\D$, $H(\mathcal{F}_1)=H(\mathcal{F}_2)$ and
  $${\bf c}[\mathcal{F}_1](p_0,p_1,p_\infty)={\bf c}[\mathcal{F}_2](p_0,p_1,p_\infty)$$
for some (and hence for all) choice of three distinct points $p_0,p_1,p_\infty\in E$.
\end{thm}

 As in the homogeneous case, the parameters $(c_{ij})$ are independent of the holonomy. Indeed, if we fix a point $(c_{ij})\in \mathbb{C}^{\frac{n(n-1)}{2}}$ and consider the map sending each admisible formal degree $n$ Weierstrass polynomial $W$ to the (formal) generators of the holonomy of $\mathcal{F}_{W,(c_{ij})}$ , we obtain a bijection onto the space of formal rotations about the points defined by $j^n(W)$ on $E$. We do not know if the preimage of a convergent element by this map is convergent in general. This would be enough to prove the convergence of the normal forms of Theorem \ref{t:formalnormalforms}.

The condition of homogeneity of a foliation in $\mathcal{D}$ can be read in the jet of order $2n+1$ of any formal one-form $\omega$ representing it. In particular we can identify the set of homogeneous dicritical foliations  by the jet of order $2n+1$ of its normal form. 
Next we give a geometric interpretation for some of the invariants obtained. They are related to the invariants obtained in the homogeneous case.

For any $k\in \mathbb{N}$ and $A\subset E$,  let $\textup{div}(E\setminus A)(k)$ denote the set of \emph{positive} divisors in $E\setminus A$ of degree $k$. If $A=\emptyset$, this space is equivalent to the projectivisation of the space of homogeneous polynomials in two variables of degree $k$, which has the structure of $\mathbb{P}^k$.
Given $\mathcal{F}\in\mathcal{D}(n)$, we can define a subset of $\textup{div}E(n+3)$ associated to $\mathcal{F}$ as follows:
\begin{equation}
  \textup{div}(\mathcal{F})=\{\textup{Tang}(\mathcal{F},\mathcal{G})_{|E}\in\textup{div}E(n+3): \mathcal{G}\in\mathcal{D}(0)\}\label{eq:div[F]}
\end{equation}
It is a linear affine subspace of $\textup{div} E(n+3)$ of dimension four, regardless of $n$. By construction $\text{div}(\mathcal{F})\subset\text{div}(E\setminus |T(\mathcal{F})|)$. If $\mathcal{F}_1$ and $\mathcal{F}_2$ are $E$-equivalent then obviously $\textup{div}(\mathcal{F}_1)=\textup{div}(\mathcal{F}_2)$.
There exists an equivalence relation $\sim_{T(\mathcal{F})}$ on $\textup{div}E(n+3)$ depending only on $T(\mathcal{F})$ whose classes correspond precisely to a subset of the form (\ref{eq:div[F]}). Indeed, given a divisor \[D=r_1p_1+\cdots+r_kp_k\in\textup{div}E(n)\] we say that two divisors $D_1,D_2\in\textup{div}E(n+3)$ are $D$-equivalent, and we denote it by $D_1 \sim_D D_2$, if there exist homogeneous polynomials in two variables $P_1,P_2$ and $R$ defining $D_1,D_2$ and $D$ respectively satisfying that $P_2=P_1+RQ$ for some homogenoeus polynomial $Q$ of degree $3$. On the other hand, denote by $\textup{Rot}(D)$ the set of $k$-uples $(h_1,\ldots,h_k)$ where each $h_i:(E,p_i)\rightarrow (E,p_i)$ is a holomorphic germ locally conjugated to the rotation by angle $2\pi/(r_i+1)$ and $${\bf E}=\bigcup_{D\in\textup{div}E}{\bf E}_D\textup{ where }{\bf E}_D=\frac{\textup{div}(E\setminus |D|)(\textup{deg}(D)+3)}{\sim_D}\times\textup{Rot}(D)$$ with its natural projection ${\bf E}\rightarrow \textup{div}E$. We get a natural map $\mathfrak{I}:\mathcal{D}/\sim_E\rightarrow {\bf E}$ given by \[ \mathfrak{I}\left(\left[\mathcal{F}\right]\right)= ([\textup{div}(\mathcal{F})],H(\mathcal{F}))\in{\bf E}_{T(\mathcal{F})}\]
which is well defined since $E$-equivalences fix any point on $E$, where all invariants we deal with are computed.

\begin{thm}\label{thm:invariantsinD}
 The map $\mathfrak{I}$ is onto. Its fiber over an element in ${\bf  E}_D$ is biholomorphic to $\mathbb{C}^M$ where $M=\max \left(0,\frac{(\deg D-2)(\deg D-1)}{2}\right)$.
\end{thm}

This Theorem also proves that all the invariants in Theorems \ref{t:invariantsinDR} and \ref{t:invariants} are realized by foliations in $\mathcal{D}$.

\corr{ To describe the space $\mathcal{D}/\sim$ we just need to remark that there is a natural action of the group $GL(2,\mathbb{C})$ on $\mathcal{D}/\sim_E$ once we have chosen coordinates $(x,y)$ in $(\mathbb{C}^2,0)$. The action associates to each matrix the natural linear transformation in the two variables $(x,y)$. The quotient is precisely $\mathcal{D}/\sim$.   This action preserves fibers of $\mathfrak{I}$, and actually the map $\mathfrak{I}$ is equivariant with respect to the natural homomorphism $GL(2,\mathbb{C})\rightarrow \text{Aut}(E)$. Hence we can define a surjective map $\overline{\mathfrak{I}}:\mathcal{D}/\sim\rightarrow {\bf E}/\text{Aut}(E)$ whose fibers are as in Theorem \ref{thm:invariantsinD}.}

The paper is organized as follows: in Section 2, we treat all Theorems concerning single homogeneous dicritical foliations. Section 3 is devoted to unfoldings of homogeneous foliations. At last, in Section 4, we develop the formal arguments to prove the classification results in $\mathcal{D}$.

We are thankful to G. Casale, J-F. Mattei, L. Ortiz, E. Paul, E. Rosales, P. Sad,  E. Salem and L. Teyssier for useful conversations on the subject of the paper.

\section{Classification in $\DR$}

\subsection{Compactication of separatrices of homogeneous foliations}
In this subsection, we give the construction for the proof of Theorem \ref{t:construction}. It is based on an idea of F. Loray (see \cite{Lo}) of extending germs of foliations along some separatrix by compactifying the leaf to a rational curve in some foliated complex surface. The hypothesis on the homogeneity of the foliations will allow us to make adequate choices for the extended foliations.

Given an holomorphic regular foliation $\mathcal{F}$ around an embedded curve $C$ in a complex surface $S$, we define the tangency divisor $T(\mathcal{F})=\text{Tang}(\mathcal{F},C)$, its holonomy $H(\mathcal{F})=\bigsqcup _{p\in|T(\mathcal{F})|}H(\mathcal{F},p)$ and the holonomy class $H[\mathcal{F}]$ as the class of $H(\mathcal{F})$ modulo automorphisms of $C$.

Given a foliation $\mathcal{F}\in\DR (N)$, we consider a radial foliation $\mathcal{G}$ such that
\[\text{Tang}(\mathcal{F},\mathcal{G})=n_1L_{p_1}+\ldots n_kL_{p_k}\] where $L_{p_i}$ is the leaf of $\mathcal{F}$ through the point $p_i\in E\setminus
|T(\mathcal{F})|$. By construction, we can find local coordinates $(u,y)$ around each point $p_i$ where $p_i=(0,0)$, $E=\{y=0\}$,  $\mathcal{G}=\{\dd u=0\}$ and for some unit $f$, $\mathcal{F}=\{\dd u+u^nf(u,y)\dd y=0\}$. The next lemma shows that we can find local normalizing coordinates for the pair $(\mathcal{F},\mathcal{G})$ around each point $p_i$.

\begin{lem}\label{ptitlemme}
Let us consider the two germs of forms
\[
\omega_{1}=\dd u+u^{n}f\left(u,y\right)\dd y\qquad\omega_{2}=\dd u+u^{n}\dd y
\]
where $n\geq 1$ and $f$ is local unit. Then the induced germs of foliations are analytically
conjugated by a conjugacy of the form $\left(u,y\right)\mapsto\left(u,y\left(\cdots\right)\right)$.\end{lem}
\begin{proof}
Since the two forms are smooth and locally transverse except along
$u=0$ and since we require the conjugacy to preserve each leaf of the fibration
$\pi:\left(u,y\right)\mapsto u$, it is uniquely determined
on a neighborhood of $\left(u,y\right)=\left(0,0\right)$ deprived
of $u=0.$ Thus, it is enough to show that this conjugacy and its inverse are bounded
near $\left(0,0\right)$ and apply Riemann's extension Theorem to
conclude that the extension is a biholomorphism.

\noindent Let us describe the conjugacy on a neighborhood of $\left(u,y\right)=\left(0,0\right)$
deprived of $u=0.$ To do so, we will interpret its restriction to a fibre as a composition of two holonomy maps: the first going from a fibre to $u=0$ via $\omega_1$ and the second from $u=0$ to the fibre via $\omega_2$.  To get bounds, we consider a point $\left(\alpha,0\right)$
and follow the leaf of $\omega_{1}$ until one reaches the fiber $\pi^{-1}\left(u\right).$
Denote by $\left(u,\phi\left(\alpha,u\right)\right)$ the reached
point. To compute $\phi$, we consider the Cauchy system defined by
\[
\left\{ \begin{array}{c}
y\left(0\right)=0\\
y^{\prime}\left(t\right)u\left(t,\alpha\right)^{n}f\left(u\left(t,\alpha\right),y\left(t\right)\right)+u^{\prime}\left(t,\alpha\right)=0
\end{array}\right.\mbox{where }u\left(t,\alpha\right)=\left(1-t\right)\alpha+tu.
\]
Obviously, one has $y\left(1\right)=\phi\left(\alpha,u\right).$ Now
since $f$ is a local unit, we can write
\[
\left|y^{\prime}\left(t\right)\right|=\frac{\left|\alpha-u\right|}{\left|u\left(t,\alpha\right)^{n}f\left(u\left(t,\alpha\right),y\left(t\right)\right)\right|}\leq C\frac{\left|\alpha-u\right|}{\left|u\left(t,\alpha\right)\right|^{n}},
\]
for some constants $C$ and $c$. Thus, integrating, we obtain
\begin{equation}\label{eq:ineqholonomy}
\left|\phi\left(\alpha,u\right)\right| \leq C\left|\alpha-u\right|\int_{0}^{1}\frac{\dd t}{\left|u\left(t,\alpha\right)\right|^{n}}.
\end{equation}
Now, in the same way, for any point $\left(u,y\right)$ near $\left(0,0\right),$
we denote $\psi\left(u,y\right)$ the function such that the leaf
of $\omega_{2}$ reached the point $\left(\psi\left(u,y\right),0\right).$
Using that $\omega_{2}$ admits a first integral, we obtain
that
\[
\psi\left(u,y\right)=\frac{u}{\left(1-\left(n-1\right)yu^{n-1}\right)^{\frac{1}{n-1}}}\text{ if } n\geq 2 \text{ and }\psi(u,y)=ue^{y}\text{ if } n=1.
\]

The second factor of the conjugacy between the foliations induced by $\omega_{1}$ and $\omega_{2}$
is written $\phi\left(\psi\left(u,y\right),u\right)$. Using upper bounds on the expressions $|\frac{\psi}{u}-1|$ and $|\frac{u(t,\psi)}{u}|$ substituted in (\ref{eq:ineqholonomy}), we obtain a constant $M>0$ such that $$\left|\phi\left(\psi\left(u,y\right),u\right)\right|\leq M|y|$$ for all $(u,y)$ in some neighbourhood of $(0,0)$ deprived of $u= 0$. Applying the same argument to the reverse situation, following first the leaves of $\omega_1$ then these of $\omega_2$ yields a lower bound
$$m|y|\leq\left|\phi\left(\psi\left(u,y\right),u\right)\right|.$$
The combination of these two inequalities proves the lemma.

\end{proof}

\noindent Notice that if $n\neq m$ then
\[
\dd u+u^{n}\dd y\qquad \dd u+u^{m}\dd y
\]
cannot be conjugated by a conjugacy which preserves the fibration $(u,y)\mapsto u$.

Taking two germs of regular transverse foliations $(\mathcal{R}_1,\mathcal{R}_2)$ at a point $p$, if we consider the blow-up at $p$ and denote by $E_1$ the exceptional divisor and by $\mathcal{R}^1_i$ the saturated foliation around $E_1$ obtained from $\mathcal{R}_i$, then we create a locus of tangency
\[\text{Tang}(\mathcal{R}^1_1,\mathcal{R}^1_2)= E_1\]
Blowing-up again a point in $E_1$ that is regular for both foliations, we obtain a second divisor $E_2$, two foliations $\mathcal{R}^2_1,\ \mathcal{R}^2_2$ satisfying $\text{Tang}(\mathcal{R}^2_1,\mathcal{R}^2_2)= E_1+2E_2$. Inductively, we can produce a pair of foliations $(\mathcal{R}^n_1,\mathcal{R}^n_2)$ in a neighbourhood of a chain of $n$ rational curves $E_1,\ldots, E_n$ satisfying $$\text{Tang}(\mathcal{R}^n_1,\mathcal{R}^n_2)= E_1+2E_2+\ldots+nE_n.$$
By construction, around any regular point $p\in E_n$ of the $\mathcal{R}^n_i$'s we can find coordinates $(u,y)$ where $p=(0,0)$, $E_n=\{u=0\}$, $\mathcal{R}^n_1$ is given by $\dd u=0$ and $\mathcal{R}^n_2$ by $\dd u+u^ng(u,v)\dd y =0$ for some unit $g$.

Coming back to our initial pair of foliations $(\mathcal{F},\mathcal{G})$, thanks to Lemma \ref{ptitlemme}, we can glue the pair of foliations around each common separatrix $L_{p_i}$ to the pair of foliations around a chain of rational curves of length $n_i$ coming from blowing up $n_i$ times a pair of regular transverse foliations at a point $p$. We thus obtain a pair of foliations around a divisor with $1+n_1+\ldots+n_k$ rational curves. The original foliation $\mathcal{F}$ is analytically equivalent to the restriction of this foliation to the neighbourhood of the initial divisor $E$. The divisors that have been added can be now contracted. Since, at each step, we contract a component that cuts the original divisor $E$, we get at the end of the contraction a rational curve $C$ embedded with self-intersection \[-1+(n_1+\ldots n_k)=-1+(N+3)=N+2\] in a complex surface $S$. In its neighbourhood, we get two regular foliations $(\mathcal{F}_S,\mathcal{G}_S)$ that are transverse at all points of $E$. By construction $\mathcal{G}_S$ is also transverse to $E$ at all points. On the other hand, $\mathcal{F}_S$ has tangency divisor $T(\mathcal{F}_S)=T(\mathcal{F})$ and since the contractions and blow ups are done outside $|T(\mathcal{F})|$ the holonomy is preserved \[H(\mathcal{F}_S)=H(\mathcal{F}).\]This proves the first part of the claim in Theorem \ref{t:construction}: $\mathcal{F}$ is analytically equivalent to a germ obtained from $\mathcal{F}_S$ by a sequence of blow-ups, a restriction and a contraction. As for uniqueness of this model $(\mathcal{F}_S,\mathcal{G}_S)$, we have the following

\begin{lem}
   If $(\mathcal{F},\mathcal{G})$ is a pair of germs of regular foliations around a smooth rational curve $C$ embedded in a complex surface with  $C\cdot C=N+2> 2$, the degree of $T(\mathcal{F})$ is $N$ and $\mathcal{G}$ is transverse to $C$, then the foliations $\mathcal{F}$ and $\mathcal{G}$ are transverse around $C$.

Given another pair $(\mathcal{F}',\mathcal{G}')$ with the same properties around a rational curve $C'$ in a complex surface, there exists a biholomorphism between two neighbourhoods sending the pair $(\mathcal{F},\mathcal{G})$ to $(\mathcal{F}',\mathcal{G}')$ if and only if $H[\mathcal{F}]=H[\mathcal{F}']$.
\end{lem}
\begin{proof}
We choose an open covering $\left\{U_i\right\}$ of $C$ in the surface, holomorphic vector fields $v_i$ on $U_i$ generating $\mathcal{F}$ and holomorphic one forms $\omega_i$ generating $\mathcal{G}$ in the neighborhood of $C$. On the intersection $U_i\cap U_j$, we have
\begin{eqnarray*}
v_i=\phi_{ij} v_j \\
\omega_i=\psi_{ij} \omega_j
\end{eqnarray*}
where $\phi_{ij}$ and $\psi_{ij}$ are cocycles representing respectively the line bundles $T^{*}_{\mathcal{F}}$ and $N_{\mathcal{G}}$. Therefore, the contraction $\omega_i(v_i)$ is a section of $\left[T^{*}_{\mathcal{F}}\otimes N_{\mathcal{G}}\right]_C$ since
\[\omega_i(v_i)=\psi_{ij}\phi_{ij} \omega_j(v_j). \]
 Now, this section vanishes along $C$ at the point where, precisely, $\mathcal{F}$ and $\mathcal{G}$ are tangent, thus
\[\textup{Tang}(\mathcal{F},\mathcal{G})=\textup{deg}\left[T^{*}_{\mathcal{F}}\otimes N_{\mathcal{G}}\right]_C=-T_{\mathcal{F}}\cdot C+ N_{\mathcal{G}}\cdot C \]
Using the formula of Brunella \cite{Br} yields to
\begin{eqnarray*}
\textup{Tang}(\mathcal{F},\mathcal{G})&=&-C\cdot C+\textup{Tang}(\mathcal{F},C)+\mathcal{X}(C)+\textup{Tang}(\mathcal{G},C) \\
&=& -(n+2)+n+2+0=0.
\end{eqnarray*}

The sufficiency part of the second statement of the lemma is obvious. For the necessity,
   an equivalence $\phi:C\rightarrow C'$ between the holonomies of $\mathcal{F}$ and $\mathcal{F}'$ tells us which leaf of $\mathcal{F}$ goes to which of $\mathcal{F}'$. On the other hand we can use the same equivalence to tell which leaf of $\mathcal{G}$ goes to which of $\mathcal{G}'$, since these foliations do not impose compatibility conditions by transversality at all points of the rational curves. By transversality of the pairs of foliations, this equivalence can be extended to the neighbourhoods uniquely if we impose that the pair $(\mathcal{F},\mathcal{G})$ is mapped to the pair $(\mathcal{F}',\mathcal{G}')$.
\end{proof}

From the previous arguments, we prove Corollary \ref{c:birational tranfs} by finding a biholomorphism between the models we have just constructed. The resulting composition of biholomorphisms, contractions and blow-ups might have indeterminacies. They lie on an invariant set that contains the separatrices over points that are blown-up to obtain one of the foliations, but not blown-up to obtain the other. When all those points coincide, the birational map does not have indeterminacies and it extends to a biholomorphism. This is the main idea behind Theorem \ref{t:invariantsinDR} that we proceed to prove.


\begin{pfof}{Theorem \ref{t:invariantsinDR}} Suppose first that $\mathcal{F}_1\in\DR(n)$ and $\mathcal{F}_2$ are equivalent via $\varphi\in\text{Aut}(\mathbb{C}^2,0)$. By blowing up source and target of $\varphi$ once, we get that the lift of $\varphi$ extends to $E$ as a biholomorphism in a neighbourhood of $E$ that sends the saturation of the pull-back foliation $\widetilde{\mathcal{F}_1}$ to the pull back foliation $\widetilde{\mathcal{F}_2}$. In particular, $\mathcal{F}_2$ has no singular points. Denote by $\phi\in\text{Aut}(E)$ the restriction of this biholomorphism to $E$. By construction, the holonomies satisfy \[\phi\circ H(\mathcal{F}_1)\circ \phi^{-1}=H(\mathcal{F}_2).\] Moreover, if $D\in\text{div}(\mathcal{F}_1)$, then there exists a radial foliation $\mathcal{G}_1$ such that $\text{Tang}(\mathcal{F}_1,\mathcal{G}_1)$ is invariant by $\mathcal{F}_1$ and $D=\text{Tang}(\mathcal{F}_1,\mathcal{G}_1)_{|E}$. By applying $\varphi$ on both sides, we get $\text{Tang}(\varphi_*(\mathcal{F}_1),\varphi_*(\mathcal{G}_1))$ invariant by $\mathcal{F}_2=\varphi_*(\mathcal{F}_1)$ and $\phi_*(D)=\text{Tang}(\varphi_*(\mathcal{F}_1),\varphi_*(\mathcal{G}_1))_{|E}$ and since $\varphi_*(\mathcal{G}_1)$ is a radial foliation, $\phi_*(D)\in\text{div}(\mathcal{F}_2)$.

For the sufficiency, suppose $\phi$ as in the statement of the theorem exists. Let $\mathcal{G}_1$, $\mathcal{G}_2$ be radial foliations such that $S_i=\text{Tang}(\mathcal{F}_i,\mathcal{G}_i)$ is invariant by $\mathcal{F}_i$ and \begin{equation} D=\text{Tang}(\mathcal{F}_1,\mathcal{G}_1)_{|E}\text{ and } \phi_*(D)=\text{Tang}(\mathcal{F}_2,\mathcal{G}_2)_{|E}  .\end{equation} The extension $\varphi$ of $\phi$ to the neighborhood is uniquely defined and holomorphic if we impose  $$\varphi_*(\mathcal{F}_1)=\mathcal{F}_2\text{ and }\varphi_*(\mathcal{G}_1)=\mathcal{G}_2.$$ Indeed, $\phi$ indicates which leaf of $\mathcal{F}_1$ goes to which of $\mathcal{F}_2$ and which leaf of $\mathcal{G}_1$ goes to which of $\mathcal{G}_2$. Since outside $S_i$ the leaves of $\mathcal{F}_i$ and those of $\mathcal{G}_i$ intersect transversally, the extension of $\varphi$ to the complement of $S_1$ is well defined, holomorphic, with image in the complement of $S_2$. It can then be extended holomorphically to $S_1$ as a biholomorphism of a neighborhood of the exceptional divisor by using Lemma \ref{ptitlemme} to both pairs $(\mathcal{F}_i,\mathcal{G}_i)$ around $S_i$. Thus after contracting the exceptional divisor we get a biholomorphism in a neighbourhood of $0\in\mathbb{C}^2$.

\end{pfof}
Remark that Theorem \ref{t:invariantsinDR} is also true for $\mathcal{F}_1=\mathcal{F}_2$ so that it also gives the structure of the group of automorphisms of $\mathcal{F}\in\DR$.

\subsection{The spaces $\DR(n)$.}\label{DRN}
In this section, we will analyze the sets $\divR (\mathcal{F})$ for homogeneous foliations $\mathcal{F}$ of small multiplicity at the origin.  The analysis will lead us to a complete description of the moduli spaces of homogeneous dicritical foliations, thanks to Theorem \ref{t:invariantsinDR}.

Remark that when $\mathcal{F}\in\mathcal{D}$, $\divR(\mathcal{F})\subset\text{div}(\mathcal{F})$  and it is therefore a (possibly empty) subset of a four dimensional affine space. A finer statement to that of Theorem \ref{t:small multiplicity} is:

\begin{thm}\label{t:modulispacesDR}

The set $\DR(n)$ is equal to $\mathcal{D}(n)$ if and only if $n\leq 3$. Moreover, for any $\mathcal{F}\in \mathcal{D}(n)$, we have

\begin{enumerate}
  \item $n=1$, then $\divR (\mathcal{F})=\textup{div}(E\setminus|T(\mathcal{F})|)(4)$.
  \item $n=2$ then $\divR (\mathcal{F})=\textup{div}(\mathcal{F})$. Let $q(\mathcal{F})\subset\divR(\mathcal{F})$ be the set of divisors with a single point in its support.
  \begin{itemize}
  \item if $T\left(\F\right)=p_1+p_2$ , then $|q(\mathcal{F})|=5$

  \item if $T\left(\F\right)=2p_1$, then $|q(\mathcal{F})|=1$.

  \end{itemize}
  \item $n=3$ then $\divR (\mathcal{F})$ is a quadric in $\text{div}(\mathcal{F})$. The set $q(\mathcal{F})$ of divisors in $\divR(\mathcal{F})$ with a point of order at least $4$ in its support is non-empty and contains generically at most 24 elements.
  \begin{itemize}
  \item if $T\left(\F\right)=p_1+p_2+p_3$, then generically $|q(\mathcal{F})|=24$;
  \item if $T\left(\mathcal{F}\right)=2p_1+p_2$ then generically $|q(\mathcal{F})|=18$
  \item if $T\left(\F\right)=3p_1$, then generically $|q(\mathcal{F})|=6$
  \end{itemize}

  \item $n=4$, the $1-$form \[y^4(x\dd y-y\dd x)+x^6\dd x+y^7\dd y\]  belongs to $\mathcal{D}(4)$ but not to $\DR(4)$.
  \item for any $n\geq 5$ the $1-$form of $\mathcal{D}(5)$
  \[y^n(x\dd y-y\dd x)+x^{n+2}\dd x+y^2x^{n+1}\dd y\] does not belong to $\DR(5)$.
\end{enumerate}
  \end{thm}
  \corr{
In the case of $n=1$ we get a result of Cerveau (see \cite{klu}) as a consequence of Theorems \ref{t:invariantsinDR} and \ref{t:modulispacesDR} that improves the statement of Corollary \ref{c:birational tranfs}:
\begin{cor}
  Two foliations in $\mathcal{D}(1)$ are analytically equivalent if and only if they share the same holonomy class. 
\end{cor}
}

\begin{rem} \label{r:unfoldings for homogeneous} If $\mathcal{F}\in\DR (n)$ and we take a germ of transverse subvariety to $\divR(\mathcal{F})$ in $S(\mathcal{F})=\{\divR(\mathcal{F}'): \mathcal{F}'\in\DR(n)\text{ and } H(\mathcal{F}')=H(\mathcal{F})\} \subset \textup{div} E(n+3)$, we can easily construct a isoholonomic germ of deformation of $\mathcal{F}$. As we will see in section \ref{s:deformation vs unfolding} this will be enough to guarantee that the deformation underlies a germ of unfolding, and also that, for $n\leq 3$, we have coincidence of dimension of $S(\mathcal{F})$ and of $ \text{div}(E)(n+3)$. Recall that, on the other hand, the dimension $\mathcal{M}(\mathcal{F})$ of the base space of the universal equisingular unfolding of such $\mathcal{F}$ is $n(n-1)/2$ (see \cite{M}). For $n=1$, it is zero and this appears as the codimension of $\divR(\mathcal{F})$ in $\text{div}(E)(4)$ in item 1. In the cases $n=2,3$ the codimension of $\divR(\mathcal{F})$ in $\textup{div} E(n+3)$ coincides with the dimension of the base space of the universal equisingular unfolding. We can construct universal equisingular unfoldings in this way for $n=2,3$.  In the case $n=2$, we can realize any four dimensional affine subspace as some $\divR(\mathcal{F})$. For $n=3$, we need to find a point in $\divR(\mathcal{F})$ having a point in the support of order at least $3$ to realize the quadric. Some of these unfoldings can be obtained by pull-back, allowing us to obtain explicit forms in some cases treated in section \ref{s:unfoldings}.
\end{rem}

A germ of function $h$ in $\mathbb{C}^2$ is quasi-homogeneous if and only if there exists a vector field $X$ such that $X\cdot h=h$. In the same way, we obtain the following criterion for a foliation $\mathcal{F}$ to be in $\DR$. Recall that we denote by $\DR(n)$ the intersection of $\DR$ and $\mathcal{D}(n)$.

\begin{lem} Let $\omega$ be a one form  representing $\mathcal{F}\in \D(n)$.
Then $\mathcal{F}\in \DR(n)$ if and only if there exists a formal vector field $\hat{X}$ with radial linear part and a formal unit $\hat{u}$ such that
\begin{equation}\label{equationDR}
\hat X\cdot \omega(\hat X)=\hat u\omega(\hat X).
\end{equation}

\end{lem}
\begin{proof}{}
Suppose that there exist $\mathcal{G}\in \D(0)$ such that $\mbox{Tang}(\F,\G)$ is invariant by $\mathcal{F}$. Then, it is also invariant by $\G$. Now, let us consider $X$ and $\omega$ any vector field and form that represent respectively $\G$ and $\F$. The contraction $\omega(X)$ is an equation of the tangency locus. Since it is invariant by $X$, the derivative $X\cdot\omega(X)$ can be holomorphically divided by $\omega(X)$ thus there exist a function $u$ such that
\[X\cdot \omega(X)=u\omega(X).\]Looking at the multiplicity at $0$ of both component of the equality above ensures that $u$ is a unit, i.e., $u(0)\neq 0$.
Here, $X$ and $u$ are convergent, thus, also formal. Conversely, if (\ref{equationDR}) has a formal solution, then $\omega(\hat X)$ is the product of formal equations of convergent separatricies of $\mathcal{F}$. Thus, there exists a formal unit $\hat v$ and a convergent equation $F$ of separatricies of $\F$ such that
\[\omega(\hat X)=\hat vF.\] Now, we deduce the existence of a convergent solution $(X,v)$ to the above equation by Artin's Theorem (see \cite{A}) whose first jets coincide with those of $(\hat X,\hat v)$. Obviously, the convergent vector field $X$ satisfied the equation (\ref{equationDR}) for some unit $u$.

\end{proof}

\begin{lem}\label{finitedeter} Let $\omega$ be a $1$-form representing $\mathcal{F}\in\mathcal{D}(n)$.
Suppose that there exist a vector field $X$ and a unit $u$ such that
\[j^{2n+1}\left( X\cdot \omega( X)- u\omega( X)\right)=0\]

then $\mathcal{F}\in \DR(n)$.
\end{lem}
\begin{proof}
In the proof below, the notation $\square_i$ stands for the homogeneous compotent of degree $i$ of the object $\square$. Suppose that there exists a vector field $X$ and a unit $u$ as in the Lemma. We are going to modify the component of degree $k+3$ of $X$ and of degree $k+1$ of $u$ so that
\[j^{n+3+k+1}\left( X\cdot \omega( X)- u\omega( X)\right)=0.\]
That will ensure by induction the existence of a formal solution to the equation (\ref{equationDR}). Now, a straightforward computation shows that
\begin{equation}\label{jetsuivant}
  \left( X\cdot \omega( X)- u\omega( X)\right)_{n+3+k+1}= \omega_{n+1}\left(X_{k+3}+u_{k+1}X_2\right)+u_{k+1}\omega_{n+2}(X_1)+(\cdots)
  \end{equation}
Here the dots $(\cdots)$ stand for terms which depends only of components of $u$ of degree strictly smaller than $k+1$ and of components of $X$ of degree strictly smaller than $k+3$. Let us denote by the component of degree $k$ of $X$ by $X_k=A_k\partial x+B_k\partial y$. From the equation above and using the notation $\omega=R_n(x\dd y-y\dd x)+P_{n+2}\dd x+Q_{n+2}\dd y+\cdots$, we obtain
\begin{eqnarray*}
 \lefteqn{ \left( X\cdot \omega( X)- u\omega( X)\right)_{n+3+k+1}} \\
  &=& R_n\left(yA_{k+3}-xB_{k+3}+u_{k+1}(yA_2-xB_2)\right)+u_{k+1}(xP_{n+2}+yQ_{n+2})+(\cdots)
  \end{eqnarray*}
This equation can always be made equal to $0$ provided that $n-1\leq k+1$: indeed, since $R_n$ and $xP_{n+2}+yQ_{n+2}$ are relatively prime, applying the Bezout result's in $\mathbb{C}[t]$ to the dehomogenized relation above ensures the existence of a polynomial function $\tilde{u}$ of degree smaller than $n-1$ and a polynomial function $\tilde{V}$ of degree smaller than $n+3$ such that
\[\tilde{R}_n\tilde{V}+\tilde{u}\left(\tilde{P}_{n+2}+t\tilde{Q}_{n+2}\right)+\tilde{(\cdots)}=0.\] Since $n-1\leq k+1$ and $n+3\leq k+4$, it can be seen that one can find $A_{k+3}$, $B_{k+3}$ and $u_{k+1}$ such that the equation (\ref{jetsuivant}) is satisfied.
\end{proof}

\begin{cor}
The subset $\divR(\F)$ is an algebraic sub-variety of $\textup{div}(\F)$.
\end{cor}

Now, we can give the proof of the Theorem \ref{t:modulispacesDR}.

In what follows, we will denote by $S_{n+3}$ the function $xP_{n+2}+yQ_{n+2}$.

\begin{itemize}
\item For $n=1$, Lemma \ref{finitedeter} ensures that
\[ \divR(\F)=\mbox{div}(E\setminus |T(\mathcal{F})|)(4).\]
Notice also that one can argue the following way for $n=1$: any $4$-uple of smooth invariant curves can be straightened to their tangent lines by a local biholomorphism. In the new coordinates, say $(x,y)$, these four lines are invariant for the radial vector field, $R=x\partial x+y\partial y$. Since the multiplicity of the tangency locus is equal to $4$, the tangency locus is exactly these four lines. Now according to the above corollary, $\divR(\F)$ is closed in $\text{div}(\mathcal{F})$,  which gives the property.
\item For $n=2$, using Lemma \ref{finitedeter} yields $\divR(\F)=\mbox{div}(\F)$. Now, the tangent cone of the tangency locus between $\F$ and $\G\in \mathcal{D}(0)$ is written
\[\left(\omega(X))\right)_5=\underbrace{xP_4+yQ_4}_{H_5}+R_2\underbrace{\left(yA_2-xB_2\right)}_{S_3}=0.\]
This tangent cone reduces to a single multiple point if and only if there exists a unit $u$ and $\alpha\in \mathbb{C}$ such that
\[H_5+R_2S_3=u\left(x+\alpha y\right)^5\]

If $|T(\F)|$ is a single point, up to some linear change of coordinates we can suppose that $R_2=y^2$. Thus, the relation above implies that
\[\left\{\begin{array}{rcl} u&=&H_5(1,0)  \\ \alpha&=&\frac{(H_5)_{4,1}}{5H_5(1,0)} \end{array}\right.\]
where $(H_5)_{4,1}$ is the coefficient of $x^4y$ in $S_5$. If $|T(\F)|$ consists of two distinct points, up to some change of coordinates we can suppose that $R_2=y(y-1)$. In this case, the solutions are given by the system
\[\left\{\begin{array}{rcl} u&=&H_5(1,0)  \\ (1+\alpha)^5&=&\frac{H_5(1,1)}{H_5(1,0)} \end{array}\right.\] which has exactly five solutions for $H_5(1,1)\neq 0$.
\item For $n=3$, following the previous lemma, it is enough to show that there exists a solution to
\begin{equation}\label{jet7}
j^{7}\left( X\cdot \omega( X)- u\omega( X)\right)=0.
\end{equation}
We consider the following notation
\[\omega= R_3(x\dd y-y\dd x)+P_5\dd x+Q_5\dd y+P_6\dd x+Q_6\dd y+\cdots.\]
Below, we are only going to consider the generic case, that is, when $R$ has three distinct points in its tangent cone
\[R=(y-\tau_1 x)(y-\tau_2 x)(y-\tau_3 x).\]
Notice that we can suppose that $\tau_1\neq 0$.
Up to some multiplication by a unit, we can furthermore suppose that the vector field $X$ is written
\[X=x\partial x+y\partial y+A_2\partial x +\delta_4 x^2\partial y+A_3\partial x+B_3\partial y\]
where $A_2$ is written in a Lagrange form
\[A_2(x,y)=\delta_1\frac{(y-\tau_2 x)(y-\tau_3 x)}{(\tau_1-\tau_2 )(\tau_1-\tau_3)}+\delta_2\frac{(y-\tau_1 x)(y-\tau_3 x)}{(\tau_2-\tau_1 )(\tau_2-\tau_3)}+\delta_3\frac{(y-\tau_1 x)(y-\tau_2 x)}{(\tau_3-\tau_1 )(\tau_3-\tau_2)}.\]
Finally, we set $u=6+u_1+\cdots$ where $u_1=u_{10}x+u_{01}y$. Here, the unknown variables are the $\delta$'s, the two coefficients of $u_1$ and the coefficient of $A_3$ and $B_3$. If we denote $\omega(X)=H_6+H_7+\cdots$, the initial equation (\ref{jet7}) is written
\begin{equation}\label{jet72}
H_7+u_1H_6-A_2\partial_xH_6-\delta_4x^2\partial_yH_6=0.
\end{equation}
Now, since $H_7$ is written
\[H_7=xP_6+yQ_6+P_5A_2+Q_5\delta_4x^2+R_3(yA_3-xB_3)\] and thus contains $A_3$ and $B_3$ as free and linear parameters, the equation (\ref{jet72}) has a solution if and only if the evaluation at each point $(1,\tau_i)$ of (\ref{jet72}) which are the roots of $R_3$ yields $0$. After a straightforward computation, we are led to a system of three equations that are written
\[\delta_i^2\frac{\partial R_3}{\partial x}(1,\tau_i)+u_{10}S_6(1,\tau_i)+u_{01}\tau_iS_6(1,\tau_i)+L_i\left(\{\delta_j\}_{j=1..4}\right) =0\quad i=1,2,3\]
where the function $L_i$ are linear functions of $\delta_1,\ \delta_2$ and $\delta_3$ and quadratic in $\delta_4$.
The last two equations $i=2, 3$ can be seen as a linear system in $u_{10}$ and $u_{01}$ whose determinant is
\[\left|\begin{array}{cc} S_6(1,\tau_2) & \tau_2 S_6(1,\tau_2)\\S_6(1,\tau_3) & \tau_3 S_6(1,\tau_3)\end{array}\right| =S_6(1,\tau_2)S_6(1,\tau_3)(\tau_2-\tau_3)\]which is not equal to $0$ because $S_6$ and $R_3$ have no common roots and $\tau_2\neq \tau_3$. Thus, $u_{10}$ and $u_{01}$ can be substitute in the first equation which can be solved because the coefficient of the quadratic term $\delta_1^2$ is equal to $\frac{\partial R_3}{\partial x}(1,\tau_1)\neq 0$. If $R_3$ has for instance a double roots, say $\tau_1$, then the second equation is replace by the partial derivative of (\ref{jet72}) with respect to $y$ applied to $(1,\tau_1)$ which has also to be $0$. Then, the computations are much the same as above. In any case, the final equation is quadratic in the variables $\delta_i$ and thus $\divR(\F)$ is a quadric.
\bigskip

Now, a point of coordinates $(1,t)$ which is in the support of some $\divR(\F)$ is of multiplicity 4 if and only if \[H_6^{(i)}(1,t)=0,\quad i=0,\ldots,3. \]
Let us write
\[H_6(1,t)=S_6(1,t)+R_3(1,t)(\alpha_0+\alpha_1t+\alpha_2t^2+\alpha_3t^3)\]
Notice that the coefficients $\alpha_i$ can be linearly written in terms of $\delta_i$'s. The previous equations can be written
\[\left\{ \begin{array}{rcl} \alpha_0+\alpha_1t+\alpha_2t^2+\alpha_3t^3 &=& -\left(\frac{S_6}{R_3}\right)(1,t) \\
 \alpha_1+2\alpha_2t+3\alpha_3t^2 &=& -\left(\frac{S_6}{R_3}\right)^{(1)}(1,t) \\
 2\alpha_2+6\alpha_3t &=& -\left(\frac{S_6}{R_3}\right)^{(2)}(1,t) \\
 6\alpha_3 &=& -\left(\frac{S_6}{R_3}\right)^{(3)}(1,t) \\ \end{array}\right. \]
Solving the following linear system, we express each coefficient $\alpha_i$ as a rational function of the variable $t$ which appear to be of degree 0, thus,
\[\alpha_i=\frac{\cdots}{R_3^4}\] where the dots stand for some polynomial function of degree at most $12$. Now, if we substitute these expressions in the quadratic equation that defines $\divR(\F)$, we are led to a polynomial equation of degree at most $24$. To check that generically this polynomial function has degree $24$ and $24$ distinct solutions,it is enough to exhibit an example satisfying these two conditions. Using MAPLE\footnote{One can find a MAPLE worksheet doing the mention computation at \url{http://www.math.univ-toulouse.fr/~genzmer/}}, we can compute that for
\[\F: y(y^2-x^2)(x\dd y-y\dd x)+(x^5+y^5+x^4y)\dd x\]
the polynomial function is
\begin{multline*}
 \frac{10}{3}{t}^{24}+20\,{t}^{22}-{\frac {20}{3}}{t}^{21}-{\frac {190}{3}}{t}^{20}-\frac {5944}{3}{t}^{19}-5898{t}^{18}-9472{t}^{17}-{\frac {22709}{3}}{t}^{16}\\
\shoveright{-3732{t}^{15}-1008{t}^{14}-{\frac {5948}{3}}{t}^{13}-2783{t}^{12}-{\frac {8224}{3}}{t}^{11}
 -394{t}^{10} +{\frac {4616}{3}}{t}^{9}}\\
+207{t}^{8}-{\frac {2188}{3}}{t}^{7}-{\frac {40}{3}}{t}^{6}+248{t}^{5}+{\frac {251}{3}}{t}^{4}+{\frac {32}{3}}{t}^{3}-{\frac {40}{3}}{t}^{2}-{\frac {32}{3}}t-\frac{2}{3},
\end{multline*}
which has no common factor with $t(t^2-1)$ and has only simple zeros.
\bigskip

Now, the same argument can be performed when $T(\F)$ has a double or a triple point. Generically, the degree of the polynomial function obtained has above will be respectively 18 and 6.

\item For $n\geq 5$, an obstruction to solve the equation (\ref{equationDR}) appears already for the jet of order $n+4$. Indeed, if we take $R=y^n$, $P_{n+2}=x^{n+2}$, $Q_{n+3}=y^2x^{n+1}$, $Q_{n+2}=P_{n+3}=0$ we have $\partial_xH_{n+3}=(n+2)x^{n+1}+y^n(\ldots)$ and $\partial_yH_{n+3}=y^{n-1}(\ldots)$. The equation on degree $n+4$ becomes
\begin{eqnarray*}
-y^3x^{n+1}=A_2x^{n+2}+y^n(\ldots)+A_2((n+2)x^{n+1}+y^n(\ldots))\\
+B_2(y^{n-1}(\ldots))+u_{1}(x^{n+3}+y^n(\ldots))
\end{eqnarray*} which has no solution if $n\geq 5$ since there is no term in $y^3x^{n+1}$ in the right term of the above equality. Thus $\DR(n)\neq\mathcal{D}(n)$ for $n\geq 5$.

\item For $n=4$, we have also $\DR(n)\neq\mathcal{D}(n)$ but it is much more difficult to find a counter-example. Indeed, the obstruction appears only on the homogeneous term degree $n+5=9$ whereas it appears on the degree $n+4$ for any $n$ bigger than $5$. Actually the following form $$y^4(x\dd y-y\dd x)+x^6\dd x+y^7\dd y$$ does not belong to $\DR(4)$ while it belongs to $\mathcal{D}(4)$. The verification of this last claim can be made using any formal computing program, for instance \textsc{maple} but is too long to be reproduce here. However, it presents no special difficulty.

\end{itemize}

Even if we were not able to prove it, we are convinced that for $n\geq 4$, the space $\DR(n)$ has  strictly positive codimension in $\mathcal{D}(n)$. To support this claim, we remark that whether or not an element of $\mathcal{F}\in \mathcal{D}(n)$ belongs to $\DR(n)$ relies on four parameters of the affine space $\text{div}(\mathcal{F})$ which must satisfy $\frac{n(n-1)}{2}$ equations as highlighted in \cite{GP}. For $n\geq 4$, there are more conditions than parameters.

\subsection{Normal forms in $\DR$}

In this subsection, we are going to prove Theorem \ref{t:normalformsinDR}.

\begin{figure}[h!]
\center{\includegraphics[scale=0.7]{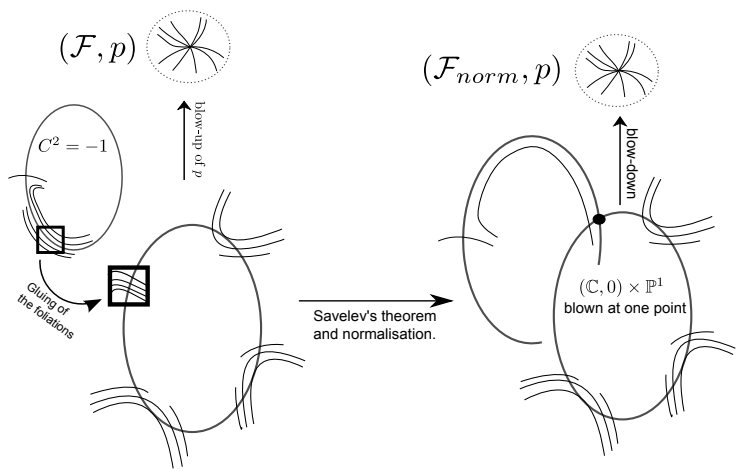}}
\caption{Construction of normal forms.}
\end{figure}

Let $\mathcal{G}$ be a radial foliation satisfying that $\text{Tang}(\mathcal{F},\mathcal{G})$ is invariant by $\mathcal{F}$ and $\text{Tang}(\mathcal{F},\mathcal{G})_{|E}=D$. Consider a coordinate $(x,y)$ in which $\mathcal{G}$ is linear such that the direction $x=0$ corresponds to the point  in the support of $D$. In this situation, the foliations $\mathcal{F}$ and $\mathcal{G}$ are respectively given by  1-forms
\[ A(x,y)(x\dd y-y\dd x)+B(x,y)\dd x \quad\text{and}\quad x\dd y-y\dd x \]
where $B$ is a homogeneous polynomial  of degree $n+2$ and $A$ is holomorphic.

Now consider  the blowing-up of the origin given in local charts by $x=uv$ and $y=v$. In a neighborhood of the leaf $u=0$, we have $\mathcal{G}$ given by $du=0$ and $\mathcal{F}$ given by $a(u,v)du+u^kb(u,v)dv$ where $a$ and $b$ are some local units and $k\geq 1$. As we did in the proof of Theorem \ref{t:construction}, we will compactify the separatrix $u=0$ for the pair $(\mathcal{F},\mathcal{G})$ around it with a convenient model by using Lemma \ref{ptitlemme}. The convenient model is the blow-up at the origin of the pair of germs of regular holomorphic foliations in two variables $(z,w)\in\mathbb{C}^2$ defined by $dw=0$ and $dw+w^{k-1}dz=0$ respectively. Indeed, the exceptional divisor is then invariant for both foliations and the tangency order between the foliations along it is $k$. By using Lemma \ref{ptitlemme}, we can glue a neighbourhood of a regular point for both foliations of this new $(-1)$-curve $C$ to the pair $(\mathcal{F},\mathcal{G})$ along the chosen separatrix. In this way, we obtain a complex bifoliated surface that is a neighbourhood of a union of two $(-1)$-curves that intersect transversely at a point. This neighbourhood is well known. Indeed, let us consider the surface $S$ obtained by blowing-up once the point $(z,x)=(0,0)$ in $\mathbb{CP}(1)\times \mathbb{D}$. We denote by $U_{-1,-1}$ any neighborhood of the union of the total transform of the divisor $x=0$. Notice that this divisor is the union of two smooth rational curves, each of self-intersection equal to $-1$, i.e. two $(-1)$-curves intersecting at a point transversely.

\begin{lem}\label{l:U-1-1}
Let $\hat{C}$ be the union of two $(-1)$-curves embedded in a complex surface that intersect transversely at one point. Then there exists a neighborhood of $\hat{C}$ that is isomorphic to some $U_{-1,-1}$.
\end{lem}

\begin{proof}
Using a classical result of Castelnuovo (see \cite{Br}), one can contract one of the $(-1)$-curves to a point. Since the $(-1)$-curves meet transversally the self-intersection of the image of the other $(-1)$-curve by the contraction map is zero. Now, the Theorem of Savelev \cite{Sav, U} ensures that there exists a biholomorphism  from a neighbourhood of this curve to $\mathbb{CP}(1)\times \mathbb{D}$ sending the curve to the divisor $x=0$ and the contraction point $p$ to $(0,0)$. This isomorphism can be lifted to the blowing-up of the source at $p$ and that of the target at $(0,0)$ thus producing the desired isomorphism.
\end{proof}

Using  Lemma \ref{l:U-1-1}, the above situation is isomorphic to a couple of foliations defined in the neighborhood surface of type $U_{-1,-1}$. If we contract the image of the curve $C$, we are led to a couple of foliations $\widetilde{\mathcal{F}}$ and $\widetilde{\mathcal{G}}$ defined on $\mathbb{CP}(1)\times \mathbb{D}$ such that $\widetilde{\mathcal{G}}$ is regular and transverse to the divisor $E_0=\mathbb{CP}(1)\times\{0\}$. The foliation $\widetilde{\mathcal{F}}$ is regular and generically transverse to $E_0$. The divisor of tangency $\text{Tang}(\widetilde{\mathcal{F}},E_0)$ coincides with $T(\mathcal{F})$ and the tangency locus $\text{Tang}(\widetilde{\mathcal{F}},\widetilde{\mathcal{G}})$ is invariant. We can choose coordinates
  $(s,t)$ in $(\mathbb{C},0)\times\mathbb{P}^1$ such that $\widetilde{\mathcal{G}}$ is given by $dt=0$ and $(0,\infty)$ is the point where the divisor was contracted. Since $\text{Tang}(\widetilde{\mathcal{F}},\widetilde{\mathcal{G}})$ is invariant by $\widetilde{\mathcal{G}}$ also, $\widetilde{\mathcal{F}}$ is given by a form $$\omega(s,t)=A(s,t)dt+Q(t)ds$$
where $Q$ is a polynomial that has its roots precisely at the common leaves of $\widetilde{\mathcal{F}}$ and  $\widetilde{\mathcal{G}}$. Thus it has degree at most $n+2$. On the other hand the function $A(s,t)$ is holomorphic in $(s,t)$ and polynomial when restricted to each fixed $s$. Since $A(0,t)$ is a polynomial of degree $n$, we have that for each $s\in(\mathbb{C},0)$ there exists a constant $u(s)\in\mathbb{C}^*$ and a unique monic polynomial of degree $n$, $r(s,t)=t^n+\hat{a}_{n-1}(s)t^{n-1}+\cdots+\hat{a}_0(s)$ such that $A(s,t)=u(s)r(s,t)$. Indeed, all the components of the divisor of tangency between $\widetilde{\mathcal{F}}$ and the fibration $ds=0$ pass through the points $(0,t_i)$ where $t_i$ is some root of $r(0,t)=0$. By defining a new variable $x$ by the relation $\dd x=\frac{ds}{u(s)}$, the foliation $\widetilde{\mathcal{F}}$ is represented in the $(x,t)$ variables of $(\mathbb{C},0)\times\mathbb{P}^1$ by  $$(t^n+a_{n-1}(x)t^{n-1}+\ldots+a_0(x))\dd t+Q(t)\dd x$$ for some holomorphic germs $a_0,\ldots,a_{n-1}\in\mathbb{C}\{x\}$.
Blowing up the point $(0,\infty)$ and contracting the strict transform of $0\times\mathbb{P}^1$ which has self-intersection $-1$ corresponds in this chart to setting $t=\frac{y}{x}$. Multiplying  the resulting expression by $x^{n+2}$ gives a holomorphic germ of one-form that has the desired form (\ref{eq:normalformDR}). This finishes the existence part of the statement in Theorem \ref{t:normalformsinDR}.

From this construction, it is clear that once we have chosen a divisor $D$ and an affine coordinate $t$ in the exceptional divisor $E$ such that $t=\infty$ corresponds to the direction $x=0$, we already get a unique polynomial $Q$ provided we impose that $Q(1,t)$ is monic. As can be deduced from Proposition \ref{p:bijection}, two foliations defined by normal forms $W_j(x,y)(xdy-ydx)+Q(x,y)dx$ for $j=1,2$ have the same holonomy if and only if they are equal. Hence the normal form with monic $Q(1,t)$, is unique up to the choices of $D$, the coordinate in $E$ and equivalences that fix each point of $E$. These equivalences can be caracterized by having linear part a multiple of the identity.

For small values of $n$, we can give a precise finite list of normal forms by using the previous argument and the choices of divisors of Theorem \ref{t:modulispacesDR}. This is the object of the following proposition.

We say that an element in $\mathbb{C}\{x\}^m$ is normalized if it is zero or  the first non-zero monomial has coefficient $1$.
\begin{cor}\label{c:normalformssmalln} For any $\mathcal{F}\in\mathcal{D}(n)$ with $n\leq 3$, there exists a finite number of convergent normal forms characterizing the analytical class of $\mathcal{F}$. They are resumed in the following table.

\begin{center}
\begin{tabular}{|p{1,2cm}|p{12cm}|l|}
\hline
$T(\mathcal{F})$& Normal Form & Nr.\\
\hline
$p$&$(y+x^3a(x))(x\dd y-y\dd x)+x^3\dd x$, $a\in\mathbb{C}\{x\}$ is normalized &$1$\\
\hline
$2p$&$(y^2+b(x)x^2y+a(x)x^3)(x\dd y-y\dd x)+x^4\dd x$ where $(a,b)\in\mathbb{C}\{x\}^2$ is normalized & $1$\\
\hline
$p_1+p_2$&$(y(y-x)+b(x)x^2y+a(x)x^3)(x\dd y-y\dd x)+x^4\dd x$ where $a,b\in\mathbb{C}\{x\}$ &10\\\hline
$3p$&$(y^3+c(x)x^2y^2+b(x)x^3y+a(x)x^5)(x\dd y-y\dd x)+(x+\lambda_1y)(x+\lambda_2y)x^3\dd x$ where $(a,b,c)\in\mathbb{C}\{x\}^3$ is normalized  and $\lambda_1,\lambda_2\in\mathbb{C}$&$\leq 6$\\ \hline
$2p_1+p_2$& $(y^2(y+x)+c(x)x^2y^2+b(x)x^3y+a(x)x^5)(x\dd y-y\dd x)+(x+\lambda_1y)(x+\lambda_2y)x^3\dd x$ where $a,b,c\in\mathbb{C}\{x\}$ and $\lambda_1,\lambda_2\in\mathbb{C}\setminus\{1\}$&$\leq 18$\\ \hline
$p_1+p_2+p_3$&$(y(y^2-x^2)+c(x)x^2y^2+b(x)x^3y+a(x)x^4)(x\dd y-y\dd x)+(x+\lambda_1y)(x+\lambda_2y)(x+\lambda_3y)^3\dd (x+\lambda_3y)$ where $a,b,c\in\mathbb{C}\{x\}$ and $\{\lambda_i\}\subset\mathbb{C}\setminus\{1,-1\}$&$\leq 144$\\

\hline
\end{tabular}
\end{center}
When a normal form is non-unique, we can deduce all the equivalent normal forms from any one of them.
 \end{cor}

\begin{pfof}{Corollary \ref{c:normalformssmalln}}
First remark that the normal forms satisfy that the contraction with the radial vector field $x\partial_x+y\partial_y$ gives a set of separatrices of the radial vector field. If $n=1,2$ there is a single separatrix, and for $n=3$ there is one separatrix of order at least four. In other words, we find elements in $q(\mathcal{F})$, the set of divisors in $\divR(\mathcal{F})$ with a single point in its support. To prove the corollary, it suffices to remark that $q(\mathcal{F})$ is non-empty for $\mathcal{F}\in \mathcal{D}(n)$ for $n\leq 3$, and use the points in $|T(\mathcal{F})|$ and the point of order at least four in any element of $q(\mathcal{F})$ to define affine coordinates in $E$. A direct application of Theorem \ref{t:normalformsinDR} gives the convergent normal form. The number of normal forms depends on the number of different affine coordinates defined by a set of geometric conditions among the points in $T(\mathcal{F})$ and $q(\mathcal{F})$.  Since the knowledge of one of the normal form defines the equations that define both sets, we can obtain all normal forms by knowing one of them easily.

In fact it suffices to distinguish three points in $E$ and identify them with points in $\mathbb{P}^1$. If $|T(\mathcal{F})|=3$ we choose them to be $0,1,-1$.
If $|T(\mathcal{F})|=2$, then we choose $0$ and $1$ to be the points in the support and $\infty$ in $q(\mathcal{F})$. If $|T(\mathcal{F})|=1$, then we take $0$ to be the point in the support and $\infty$ the multiple point in $q(\mathcal{F})$. In full generality, we lack a third point with geometric significance to define the coordinate in $E$. In that case, we notice the following: if the germ of holonomy $h:(E,p)\rightarrow(E,p)$ of order $m$ does not extend analytically to $E$, then there exists a finite set of  coordinates $t$ in $E$ such that  $h(t)=\theta t+t^k+\cdots$ with $k\geq 3$ and $\theta^{n+1}=1$, and $t=\infty$ is a point of order at least $4$ in the support of an element of $q(\mathcal{F})$. If then holonomy extends, then it is linear`. Therefore the normal form corresponds to the case $a,b,c=0$.

\end{pfof}
The convergent normal forms we have obtained look quite similar to the formal normal forms constructed by Ortiz, Rosales and Voronin in \cite{orv} and the convergent normal forms in the case $n=1$ ( see \cite{orv1}). The main difference is that we first choose a radial foliation with respect to which the given foliation is homogeneous. The convergence of its linearization map gives us the convergence of the normal forms.

\section{Unfoldings in $\DR$.}
In this section, we first give the proof of  Theorem \ref{t:unfoldings} which describes some non-trivial equisingular unfoldings for homogeneous foliations in $\DR$. Notice that, in this theorem, if the initial $1-$form is polynomial then the unfolding is also polynomial. In this way, we provide non-trivial deformations of global foliations in $\mathbb{CP}^2$ of constant topological type, a situation that is generically impossible (see \cite{IL}). In fact, each of these unfoldings compactifies to a codimension one foliation in a projective space of bigger dimension.

Then, we establish that these unfoldings are really \emph{new} examples by constructing some having no special property of integrability whereas all known examples until now presented a Louvillian first integral.

Finally, we prove Corollary \ref{c:universal unfolding} providing examples of universal equisingular unfoldings

\subsection{Proof of Theorem \ref{t:unfoldings}.}
\label{s:unfoldings}
Let us consider the rational map
\[
\Lambda\left(x,y,c\right)=\left(1+\frac{\left\langle c,y\right\rangle }{x}\right)\cdot\left(x,y\right)
\]
where $\left\langle c,y\right\rangle =\sum_{i=1}^{n-1}c_{i}y^{i}.$ By construction, it has an indeterminacy set at $x=0$ and fixes every other lines passing through the origin.

The $1$-form $\Omega$ in $\mathbb{C}^2\times \mathbb{C}^{n-1}$ as defined in the statement of Theorem \ref{t:unfoldings} is the pull-back of $\left.\Omega\right|_{c=0}=\omega $ by $\Lambda$ up to some multiplication by a meromorphic unit $u$. In
particular, $\Omega$ is an integrable $1-$form since we have
\[
u^{2}\Omega\wedge \dd\Omega=\Lambda^{*}\left(\left.\Omega\right|_{c=0}\wedge\left.\mbox{d}\Omega\right|_{c=0}\right)=0.
\]
After the blow-up $E$ of the singular locus $\left\{ x=0,y=0\right\} $, $\Omega$ is written in the coordinates of the blow-up $y=xt$
\[
E^{*}\Omega=\left(R\left(1,t\right)+\sum_{i=1}^{n-1}a_{i}\left(x+\left\langle c,tx\right\rangle \right)t^{i}\right)\mbox{d}t+Q\left(1,t\right)\dd\left(x+\left\langle c,tx\right\rangle \right).
\]
The induced foliation restricted to a fibre of $(x,y,c)\mapsto c$ over a point $c=(c_1,\ldots,c_{n-1})$ such that $1/c_1$ is not a root of $R(1,t)=0$ lies in $\mathcal{D}(n)$.
The tangency locus with the exceptional divisor $x=0$ is equal $\left\{ x=0,t=t_{i}\right\} $
where $t_{i}$ is a solution to of $R\left(1,t\right)=0.$ Since the curves ${\{x=0,t=t_{i}\}}=0$ are contained in a invariant hypersurface of $E^{*}\Omega$, the $1-$form $\Omega$ defines an equireducible unfolding of $\left.\Omega\right|_{c=c_0}$ for any $c_0$ lying in the Zariski open set $U=\{c\in\mathbb{C}^{n-1}: c_1t_i\neq 1, \forall t_i\}$

Suppose now that $\Omega$ is trivial along a certain smooth submanifold
of the space of parameter. Then, there exists a germ of application
\[
c:t\in\mathbb{C}\to\left(c_{1}\left(t\right),c_{2}\left(t\right),\cdots,c_{p}\left(t\right)\right)
\]
with $c\left(0\right)=0$ and $c^{\prime}\left(0\right)\neq\left(0,\cdots,0\right)$
such that $\left.\Omega\right|_{c\left(t\right)}$ is a trivial unfolding
of one variable. Now $\left.\Omega\right|_{c\left(t\right)}$ is written
\begin{eqnarray*}
\left(R\left(x,y\right)+\sum_{i=1}^{n-1}a_{i}\left(x+\left\langle c_{i}\left(t\right),y\right\rangle \right)y^{i}x^{n-i}\right)\omega_{R}+Q\left(x,y\right) \dd\left(x+\left\langle c_{i}\left(t\right),y\right\rangle \right)=\\
\qquad\left(\cdots\right)\omega_{R}+Q\left(x,y\right)\dd x+Q\left(x,y\right)\sum_{i=1}^{n-1}c_{i}\left(t\right)iy^{i-1}\dd y+\left(Q\left(x,y\right)\sum_{i=1}^{n-1}c_{i}^{\prime}\left(t\right)y^{i}\right)\dd t
\end{eqnarray*}
Following \cite{M}, the triviality of $\left.\Omega\right|_{c\left(t\right)}$ implies that the coefficient of $\dd t$ belongs
to the ideal generated by the coefficients of $\dd x$ and $\dd y$. If there
exists such a relation, we can evaluate it for $t=0$ and find polynomial functions $A$ and $B$ such that
\begin{equation}\label{superbieneq}
PQ=A\left(Q-yW\right)+xBW.
\end{equation}
where $P\left(y\right)=\sum_{i=1}^{n-1}c_{i}^{\prime}\left(0\right)y^{i}$. Since, $Q$ and $W$ are relatively prime, then there exists $\Delta$ such that $xB-yA=\Delta Q$ and thus $A=P-W\Delta$. Therefore, rewriting (\ref{superbieneq}) yields
\begin{equation}\label{eq:jetnunfolding}
-yP-Q\Delta+yW\Delta+xB=0
\end{equation}
Since the orders of $Q$ and $W$ at $0$ are $n+2$ and $n$, evaluating the jet of order $n$ of the above equality gives
\[-yP+x\textup{Jet}_n(B)=0\]
and thus $P=0$ which is impossible. This proves that the unfolding is non-trivial for $c_0=0$. To prove it for $c_0\in U$, it suffices to remark that the jet $n$ of equation (\ref{eq:jetnunfolding}) remains exactly the same if we impose $c(0)=c_0$ instead of $c_0(0)=0$.

\subsection{An example with no special integrability property}
As already explained, the  Theorem \ref{t:unfoldings} provides the first \emph{non trivial} examples of unfolding of singularities of foliations in $\mathbb{C}^2$. Actually, the only examples known until now were the unfoldings of singularities admitting a Louvillian first integral. Let us be more specific.: in this situation, the foliation is given by a closed form \[\omega=f_1f_2\cdots f_p\sum_{i=1}^p \lambda_i\frac{\dd f_i}{f_i}\]
where the $f_i$´s are holomorphic functions and the $\lambda_i$'s are complex numbers. In such situation, it is enough to unfold the holomorphic function $f_1f_2\cdots f_p$ to deduce an unfolding of $\omega_0$. Indeed, consider any analytical topologically trivial deformation $f_{i,\epsilon}$ of $f_i,\ f_{i,0}=f_i$ with $\epsilon\in \left(\mathbb{C},0\right).$ The one form of $\mathbb{C}^{2+1}$ defined by
 \[\Omega=f_{1,\epsilon}f_{2,\epsilon}\cdots f_{p,\epsilon}\sum_{i=1}^p \lambda_i\frac{\dd f_{i,\epsilon}}{f_{i,\epsilon}}\]is naturally closed and thus integrable. Therefore, it defines an unfolding of $\omega$.

To ensure that our theorem produces new examples, we are going to exhibit a foliation in $\DR$ that does not admit any Louvillian first integral.

Suppose  $\mathcal{F}$ lies in $\mathcal{D}(1)$ and admitting a Louvillian first integral. After the blow-up and the restriction of this first integral to $E$, we can see that $h=H(\F)$ admits a Louvillian first integral on $E$, that is there exists a non-constant holomorphic germ $f$ on $E$ such that $f\circ h=h$ and $df$ extends to a meromorphic closed one-form on $E$.  We know that such a Liouvillian function on a rational curve admits at most a countable number of singularities, that is, homotopy classes of paths along which the analytical extension of the germ of function is impossible. Below, we produce an example of periodic map $h$ such that any first integral has an uncountable number of singularities. Any foliation admitting this germ $h$ as holonomy will not admit a Liouvillian first integral.

\begin{lem}\label{l:singularity}
Let $D\subset\mathbb{C}$ be a region containing $0$ and $h:D\rightarrow D$ a holomorphic mapping, $h(0)=0$ and the germ at $0$ of $h$ satisfies $h^{\circ n}=\text{Id}$ for $n\in\mathbb{N}$. Suppose that the set of singularities of $h$ in $\partial D$ has an accumulation point $p$ and there exists a continuous extension of $h$ to a neighbourhood of $p$ in $\partial D$ satisfying $h(p)\in D$. Then any non-constant holomorphic first integral $f$ of $h$ that is defined and holomorphic on $D$ has a singularity at $p$.
\end{lem}
\begin{pf}
Before starting the proof,  remark that there is an example of an integral having the said properties, namely  $f=Id+h+h\circ h+\ldots+h\circ\ldots^{n-1}\circ h$.
Let $\gamma:[0,1]\rightarrow\overline{D}$ be a path that satisfies $\gamma^{-1}(D)=[0,1)$, $\gamma(0)=0$ and $\gamma(1)$ is a singularity of $h$.  If a first integral $f$ admits analytic extension, we denote its extension by $f_{\gamma}$ and $\lim_{t\rightarrow 1}h(\gamma(t))\in D$, then part of the graph of $h$ is contained in the set $$\{(x,y)\in\text{dom}(f_{\gamma})\times D: f(y)=f_{\gamma}(x)\}.$$
If $f'(h(\gamma(1))\neq 0$, the implicit function theorem tells us that $h$ extends analytically to $\gamma(1)$, which is not possible, so the only possibility is that for every singularity $q\in\partial D$ of $h$ where $h$ extends continuously with value in $D$ and $f$ extends analytically , $f'(h(q))=0$. Suppose for a contradiction that $f$ extends analytically to $p$. Since there exists continuous extension of $h$ to a neighbourhood of $p$ in $\partial D$, for a sequence $p_n$ of singularities of $h$ that accumulate on $p$ we have $f'(h(p_n))=0$ for all $n$. The convergence of $h(p_n)$ to $h(p)\in D$ then tells us that $f$ is constant, contrary to assumption. Hence $p$ is a singularity for $f$.
\end{pf}

To finish let us provide an example of such an $h$ with curves of singularities. Let $D$ be a simply connected plane region bounded by a Jordan curve of class $\mathcal{C}^1$ that is not analytic such that for the rotation $\theta$ of angle $2\pi/n$ around a point $p\in D$, the intersection points of $\partial D$ and $\theta(\partial D)$ are points of transversality between the curves. Let $D_p$ be the connected component of $D\cap\theta(D)$ containing $p$ and $A=\partial D_p\cap h^{-1}(D_p)$. Now, if $\varphi:\overline{\mathbb{D}}\rightarrow \overline{D}$ is the homeomorphic extension of Riemann's mapping Theorem between $\mathbb{D},0$ and $D,p$ to the boundaries given by Carath\' eodory's Theorem, then the continuous map $h:\overline{D_0}\rightarrow \overline{D_0}$ defined by $h=\varphi^{-1}\circ\theta\circ\varphi$ on $D_0=\varphi^{-1}(D_p)$, has order $n$ at $0$ and singularities at each of the points of $\partial D_0$ (see \cite{PM}, p. 628). By construction the values of the extension on $\varphi^{-1}(A)$ belong to $D_0$, so Lemma \ref{l:singularity} guarantees that the points of $\varphi^{-1}(A)$ are singular for any germ of first integral of $h$ around $0$ that is holomorphic on $D_0$.

In fact we have proven that any non-constant first integral of $h$ will not extend to any point in a curve, which is not a countable union of complex analytic sets. Hence no foliation admitting $h$ as holonomy germ will not admit a first integral that is holomorphic (and possibly multivalued) outside a countable union of  analytic sets.
\subsection{Global universal equisingular unfoldings in $\DR(2)$}
In this subsection we prove Corollary \ref{c:universal unfolding}. The form $\Omega(x,y,c)=[y^2+a(x)y(x+cy)+b(x)](x\dd y- y\dd x)+(x+cy)^4\dd x$ defined in the statement can be pulled back by the biholomorphism $(x,y,c)\mapsto (x+cy,y,c)$ to obtain $$[y^2+a(x-cy)yx+b(x)](x\dd y- y\dd x)+x^4\dd (x-cy). $$
This form corresponds precisely to the one appearing in Theorem \ref{t:unfoldings} up to a sign on $c$. Hence it is a non-trivial equisingular unfolding around each fixed parameter $c\in \mathbb{C}\setminus\{t: t^2+a(0)t+b(0)=0\}$. Since the dimension of the base space of the obtained unfolding coincides with the dimension of the universal equisingular unfolding of $\mathcal{F}_c$ , they are equivalent. It is worth mentioning that the topological type of the parameters that were excluded is not contained in $\mathcal{D}$. It is actually singular after blow-up. However these allow to compactify the foliation along the parameter space.

Let us analyze the analytic invariants of $\mathcal{F}_c$ along the parameter space. By construction $H(\mathcal{F}_c)=H(\mathcal{F}_0)$ for all admissible $c$'s. The position of the divisors $q(\mathcal{F}_c)$ defined in the proof of Theorem \ref{t:modulispacesDR} depend on $c$ by a holomorphic (possibly multivalued) non-constant function that assumes any value in $\mathbb{C}\setminus |T(\mathcal{F}_0)|$. By Theorems \ref{t:invariantsinDR} and \ref{t:modulispacesDR} we cover all analytic classes of foliations in $\mathcal{D}(2)$ having the same holonomy $H(\mathcal{F}_0)$.

It remains to see that any choice of rotations around points in $E$ is realized as the holonomy of a normal form.   Once we have a holomorphic germ in $\mathcal{D}(2)$ with given holonomy generators we can consider the coefficients $a$ and $b$ of one of its normal forms in Corollary \ref{c:normalformssmalln} as $\mathcal{F}_0$. The realization part can be found in section \ref{s:holonomy}.

\section{Classification in $\mathcal{D}$.}
In this section, we are going to prove Theorems \ref{t:formalnormalforms}, \ref{t:invariants} and  \ref{thm:invariantsinD} and their corollaries.

\subsection{Formal normal form in $\mathcal{D}$ and proof of Theorem \ref{t:formalnormalforms}}
The proof follows the lines of the proof of Theorem 4 in \cite{orv} with the appropriate changes to generalize to foliations in $\mathcal{D}$. In that paper the divisor $T(\mathcal{F})$ is assumed to have only simple points in its support, which is the generic case.

Throughout this subsection, whenever a system of coordinates $(x,y)$ is given, the radial vector field $x\partial_x+y\partial_y$ will be denoted by $\radial$. Its dual form $x\dd y-y\dd x$ is denoted by $\omega_{\radial}$. For convenience, we are going to use the vector fields rather than the $1-$forms.

Recall that two formal vector fields $V,\widetilde{V}$ in $(\mathbb{C}^2,0)$ are said to be \emph{formally E-equivalent} $V\simfor \widetilde{V}$ if there exists a formal pair of power series $\phi(x,y)=(\lambda x+\ldots,\lambda y+\ldots)$ and a formal unit $u(x,y)=\mu+\cdots$ such that $$\widetilde{V}=u\cdot\phi_*(V).$$
If the vector fields converge, $\phi$ is a formal equivalence between the induced foliations.

\textbf{Notations and conventions.}\textit{ Later in this section, $j^{m}(\square)$ and $\square_m$ stand respectively for the jet of order $m$ and the component of homogeneous degree $m$ of $\square$. Moreover, a vector field $V$ is said to have a \emph{normalized} homogeneous term of degree $n+N\geq n+2$ that is written $$V_{n+N}=P\partial_x+Q\partial_y$$ if both $P(1,t)$ and $\omega_{\radial}(V_{n+N})(1,t)$ have degree at most $n-1$ and the second has order at least $N$ at $0$.}

Let us proceed to the proof of Theorem \ref{t:formalnormalforms}.

Let $(x,y)$ be coordinates such that $x=0$, $y=x$ and $y=0$ define the points $p_0, p_1$ and $p_{\infty}$ in $E$ respectively and $V=P\partial_x+Q\partial_y$ be the holomorphic vector field representing $\mathcal{F}$ such that $P$ is a Weierstrass polynomial in $y$. Then  $j^{n+1}(V)=R\radial$ for a unique homogeneous polynomial $R(x,y)$ of degree $n\geq 1$ with $R(0,y)=y^n$. The regularity of the foliation after blow-up is equivalent to the fact that $R$ has no common factors with $\omega_{\radial}(V_{n+2})=yP_{n+2}-xQ_{n+2}$.

After applying a transformation of type $\lambda\text{Id}$ for some $\lambda\in\mathbb{C}^*$ and multiplying by an appropriate unit, we can suppose that $R(0,y)=y^n$ and $Q_{n+2}(x,0)=x^{n+2}$. In what follows, all changes of coordinates will be tangent to the identity and all units will be equal to $1$ at the origin. Thus $R$ and $Q_{n+2}(x,0)$ do not change along the next changes of coordinates.

Next we are going to normalize recursively the homogeneous components of $V$. To do so, we will use of two types of changes of coordinates. Below, we describe these changes of coordinates and compute how they affect the homogeneous components of $V$. The formal normal form will be the limiting vector field obtained after infinite recursion.
\begin{itemize}
\item  If $\phi_N=(\alpha, \beta)$ is a homogeneous vector field of degree $N\geq 2$, and we consider the vector field $\widetilde{V}=\widetilde{P}\partial_x+\widetilde{Q}\partial_y$ obtained by pushing $V$ by the transformation $\phi=\text{Id}+\phi_N+\text{h.o.t.}$ we have that
\begin{eqnarray}\label{eq:normalizetangwithradial}
j^{n+N}(\widetilde{P})&=&j^{n+N}(P)+(N-1)R\alpha-x(\alpha\partial_xR+\beta\partial_y R)\\
  j^{n+N+1}(\omega_{\radial}(\widetilde{V}))&=&j^{n+N+1}(\omega_{\radial}(V))+ (N-1)R\omega_{\radial}(\phi_N)\label{eq:radialcoord}
\end{eqnarray}

\item Moreover, if $\phi_N$ preserves the radial foliation, i.e. there exists a homogeneous polynomial $\gamma$ of degree $N-1$ such that $\alpha=x\gamma$ and $\beta=y\gamma$, then setting $\widetilde{V}^{\gamma}=(1-(N-1-n)\gamma)\widetilde{V}$, we get
\begin{eqnarray}
j^{n+N}(\widetilde{P}^{\gamma})&=&j^{n+N}(P) \\
j^{n+N+2}(\omega_{\mathcal{R}}(\widetilde{V}^{\gamma}))&=&j^{n+N+2}(\omega_{\mathcal{R}}(V))-N\gamma\omega_{\radial}(V_{n+2})
\end{eqnarray}
\end{itemize}
In particular if $V$ was normalized up to order $n+N$, so is $\widetilde{V}^{\gamma}$.
\bigskip

We will consider a sequence of equivalent vector fields $V^N$, each normalized up to order $n+N$.
Start with $V$ and find a homogeneous vector field of degree $N=2$, $\phi_2$ such that the right hand side of (\ref{eq:radialcoord}) becomes a polynomial of degree less than $n$. By dehomogeneizing the equation there is a unique possibility for $\omega_{\radial}(\phi_2)$ given by euclidean division in the ring $\mathbb{C}[t]$. The push forward of $V$ by $\text{Id}+\phi_2$ gives a vector field $\widetilde{V}$ satisfying $\omega_{\radial}(\widetilde{V})$ is normalized up to order $n+2$. Let $\eta$ be the unique homogeneous polynomial of degree $N-1$ such that $\widetilde{Q}(1,t)+\eta(1,t)R(1,t)$ is of degree less than $n$. Then $V^2=(1-\eta)\widetilde{V}$ is normalized up to order $n+2$ and $V^2\simfor V$. Suppose for induction that $N\geq 2$, $V^N$ is normalized up to order $n+N$ and $V^N\simfor V$. Let us find a $V^{N+1}\simfor V^{N}$ normalized up to order $n+N+1$.
The key ingredient for the induction is
\begin{lem}
Given a homogeneous polynomial $H$ of degree $n+N+2$ in $(x,y)$ there exists a unique polynomial $q_{n+N+2}$ in the $\mathbb{C}$- vector space $V$ generated by $y^jx^{n+2+N-j}$ for $j=N\ldots, n-1$, and homogeneous polynomials $A$ and $B$ of degrees $N+2$ and $N-1$ respectively such that $$H=q_{n+N+2}+AR+B\omega_{\radial}(V_{n+2})$$
\end{lem}

Set $H=(\omega_{\radial}(V^{N}))_{n+N+2}$ and apply the previous lemma to obtain $A,B$ and $q_{n+N+2}$. Define $\gamma=B/N$. Then defining $\widetilde{V}^{\gamma}$ as before, we get $j^{n+N+2}(\widetilde{V}^{\gamma})=q_{n+N+2}+A R$ and $j^{n+N}(\widetilde{P}^{\gamma})=j^{n+N}(P)$.  Next choose a homogeneous vector field  $\phi_N$ of degree $N$ such that $\omega_{\radial}(\phi_{N})=-A/(N+1)$. After applying $\text{Id}+\phi_N$ to $\widetilde{V}^{\gamma}$ we obtain a vector field $W^{N+1}$ with normalized $$j^{n+N+2}(\omega_{\radial}(W^{N+1})).$$  It remains to normalize the homogeneous term $P$ of degree $n+N+1$ of the $\partial_{x}$-coordinate of $W^{N+1}$. We claim that $x$ divides $P$ and thus $P(1,t)$ has degree at most $n+N$. Indeed,  By construction $x$ divides $\omega_{\radial}(W^{N+1}))_{n+N+2}$ and $$W^{N+1}_{n+N+2}-\left(\frac{(\omega_{\radial}(W^{N+1}))_{n+N+2}}{x}\right)\partial_y\quad\text{ and }\quad\radial\quad\text{ are tangent vector fields.}$$ Therefore $x$ divides $P$. By using euclidean division in $\mathbb{C}[t]$ we can find a unique homogeneous polynomial $P_{N+1}(x,y)$  of degree $N+1$  such that $P_{N+1}(1,t)$ has degree at most $n-1$, and a homogeneous polynomial $C(x,y)$ of degree $N$ such that $P=P_{N+1}+xRC$. The vector field $V^{N+1}=(1-C)W^{N+1}$ is normalized up to order $n+N+1$ and still in the same equivalence class.


The formal vector field $V^{\infty}$ satisfying $j^N(V^{\infty})=j^N(V^N)$ for all $N\geq 1$ is also formally equivalent to $V^N$ and has all its homogenous terms normalized. Its dual form has the properties stated in the statement of Theorem \ref{t:formalnormalforms}. The uniqueness of the solutions for the coefficients in the Taylor series of $V_{\infty}$ at each step  of the normalization shows that the normal form is unique. Nevertheless, the formal equivalence between the vector field $V$ and its normal form $V_{\infty}$ is non-unique since the group of automorphisms of the foliation contains the exponential of $uV$ for any unit $u$. This appears in the normalization process as the lack of uniqueness for the coefficients in the normalizing map.

\subsection{The holonomy of an element in $\mathcal{D}$ and Proofs of Theorems \ref{t:invariants} and \ref{thm:invariantsinD}}
\label{s:holonomy}
The proof of the theorem \ref{t:invariants} relies on an analysis of the holonomy of the foliations in $\mathcal{D}$. This is what is done below.

Let us denote by $\widehat{I}^{t_i}_{r_i+1}$ the set of formal series in $(t-t_i)$ that are formally conjugated to the rotation of angle $2\pi/(r_i+1)$ around $t_i$, that is to say, series \[h(t)=e^{\frac{2\pi i}{r_i+1}}(t-t_i)+\sum_{j\geq 2} h_{ij}(t-t_i)^j\text{ such that }h^{r_i+1}=\textup{Id}.\]

Given a monic polynomial $r$ of degree $n$, $r(t)=(t-t_1)^{r_1}\cdots(t-t_k)^{r_k}$ where $t_i\neq t_j$ if $i\neq j$, and  $q(t,x)\in\mathbb{C}[[t,x]]$ such that $r(t)$ and $q(t,0)$ have no common roots, define the set of formal one-forms $$\widehat{\Omega}_{r,q}=\{ (r(t)+w(t,x))dt+q(t,x)\dd x: w\in x\mathbb{C}[[x]][t]\text{ satisfies }\deg _t w<n\}.$$ For each $\omega\in\widehat{\Omega}_{r,c}$ and $i=1,\ldots,k$ define $h_i\in \widehat{I}^{t_i}_{r_i+1}$ to be the formal series in $(t-t_i)$ satisfying $\hat{f}\circ h=\hat{f}$ for some local formal first integral $\hat{f}$ of $\omega$ around $t_i$. In the case of convergent $\omega$  it coincides with the series of the generator of $H(\mathcal{F}_{\omega},t_i)$ and in general we call it the formal holonomy of the formal foliation $\mathcal{F}_{\omega}$.
\begin{prop}\label{p:bijection}
  The map $\textup{hol}_{r,q}: \widehat{\Omega}_{r,q}\rightarrow \widehat{I}^{t_1}_{r_1+1}\times\cdots \times\widehat{I}^{t_k}_{r_k+1}$ defined by $$\textup{hol}_{r,q}(\omega)=(h_1,\ldots,h_k)$$ is a bijection.
\end{prop}
\begin{pf}
  First observe that for each $h\in \widehat{I}^{t_0}_{r}$ there exists a unique formal first integral, that is a formal series $f(t)$ such that $f\circ h=f$ of the form  \begin{equation} f(t)=(t-t_0)^r+\sum_{j\geq r+1}f_j(t-t_0)^j \text{ such that } f_{j}=0\text{ for all }j\equiv 0\quad(\text{mod } r).\end{equation} that we will call \emph{normalized}. Given $(h_1,\ldots,h_k) \in \widehat{I}^{t_1}_{r_1+1}\times\cdots \times\widehat{I}^{t_k}_{r_k+1}$ we want to find a unique $\omega\in\widehat{\Omega}_{r,q}$ satisfying $\mbox{hol}_{r,q}(\omega)=(h_1,\ldots,h_k)$. If such $\omega$ exists, we can consider $k$ formal local first integrals $f^j=\sum x^if^j_i(t)\in \mathbb{C}[[x,t-t_j]]$ for the foliation near $t=t_j$ such that $f^j_0(t)$ is normalized as defined above. Below, we are going to construct inductively both $\omega$ and its $k$ formal first integrals. By definition, around the point $t_j$ we have
  \begin{equation}\label{eq:normalfirstintegral}
    0\equiv \omega\wedge df^j=\big(r(t)+w(t,x)\big) \big(\sum_{i\geq 1} if^j_i(t)x^i\big)-q(t,x)\big(\sum_{i\geq 0} \frac{df^j_i}{dt}(t)x^i\big)
  \end{equation}
    When evaluating (\ref{eq:normalfirstintegral}) on $x=0$, we get $r(t)f^j_1(t)=q(t,0)\frac{df^j_0}{dt}(t)$. Since $r(t)$ and $\frac{df^j_0}{dt}(t)$ have the same order at $t=t_j$, $f^j_1$ is a power series satisfying \[f^j_1(t_j)\neq 0.\]
    The coefficient on $x$ of (\ref{eq:normalfirstintegral}) is \begin{equation}\label{eq:xcoef first integral}
      0\equiv r(t)2f_2^j(t)+w_1(t)f^j_1(t)-\ldots\end{equation} where the dots refer to terms depending only on $f_0,f_1$ and $q$. Now since $(t-t_j)^{r_j}$ divides $r$ and $f^j_1(t_j)\neq 0$, the values of $w_1(t_j),\ldots,(w_1)^{(r_j-1)}(t_j)$ do not depend on $f_2$, but only on $f_0,f_1$ and $q$. If we add up all the conditions, we have to find a polynomial of degree at most $n-1$ determined by $n=\sum r_j$ conditions, so there is a unique possibility for $w_1$. With $w_1$ at hand, we can define $f^j_2$ for all $j$ by using (\ref{eq:xcoef first integral}). By induction, suppose we know $\{f^j_1,\ldots,f^j_l\}$ and $\{w_1,\ldots,w_{l-1}\}$, let us find $f^j_{l+1}$ for all $j$ and $w_l$. The coefficient on $x^l$ of (\ref{eq:normalfirstintegral}) is given by
      \begin{equation}\label{eq:xLcoef first integral}
        0\equiv r(l+1)f^j_{l+1}+w_lf_1^j+\ldots
      \end{equation}
    where the dots stands for an expression depending only on known functions.  We get $n$ conditions on $w_l$ that define it uniquely and independently of $f^j_{l+1}$ by considering equation (\ref{eq:xLcoef first integral}) and its derivatives up to order $r_j$ evaluated on $t_j$. With $w_l$ at hand, we define $f^j_{l+1}$ for all $j$ by using (\ref{eq:xLcoef first integral}), thus proving the induction step.
\end{pf}

\begin{rem}
By using the previous results, we can provide some more examples of universal equisingular unfoldings of elements in $\mathcal{D}$ as it has been done in theorem \ref{t:unfoldings}. Indeed, if $r(t)$ is a polynomial of degree $n$, the universal unfolding of the germ of foliation $\mathcal{F}\in\mathcal{D}(n)$ defined by
\begin{equation}\label{martineklu}
\nu=x^{n+2}\dd \left(r\left(\frac{y}{x}\right)+x\right)=0
\end{equation} is given by pull-back of $\nu$ via the rational map $$\Lambda(x,y,(c_{ij}))=(x,y)\cdot\left(1+\sum_{j=1, i<j}^{n-1} c_{ij}x^{i-j}y^{j}\right)$$
where $(c_{ij})\in\mathbb{C}^{n(n-1)/2}$. Remark that for each fixed non-zero parameter $(c_{ij})$, $\Lambda$ is an automorphism of each leaf of the radial foliation $x\partial_x+y\partial_y$ except for the leaf $x=0$, where $\Lambda$ has its indeterminacy set. To see that the resulting unfolding is equisingular and non-trivial in any direction in the parameter space, it suffices to remark that after blowing up via $y=tx$, the underlying deformation is written in dual form as $$\left( r(t)+\sum_{j=1, i<j}^{n-1} jc_{ij}x^{i+1}t^{j-1}\right)\dd t+\left(1+\sum_{j=1, i<j}^{n-1} (i+1)c_{ij}x^{i}t^{j-1}\right)\dd x.$$
 Hence for any fixed parameter we have already one of the  normal forms given in Theorem \ref{t:formalnormalforms}. By uniqueness of the normal form, any two elements in this family belong to different classes of  $E$-equivalence. Actually, the so constructed unfolding turns out to be the \emph{universal} unfolding of $\nu$ since the dimension of its space of unfoldings is $n(n-1)/2$. Notice that unfortunately, such a procedure fails when the initial form $\nu$ is in normal form and has some other non-zero terms.
\end{rem}

As a corollary of the above proposition, we obtain the
\begin{pfof}{Theorem \ref{t:invariants}} Consider two foliations $\mathcal{F}_1$ and $\mathcal{F}_2$ as in the statement of the theorem.
Their associated formal normal forms
\begin{eqnarray*}
\mathcal{F}_1\hat{\sim}_E \underbrace{(R(x,y)+\cdots)}_{W_1}(x\dd y-y\dd x)+Q_1x^3\dd x\\
\mathcal{F}_2\hat{\sim}_E \underbrace{(R(x,y)+\cdots)}_{W_2}(x\dd y-y\dd x)+Q_2x^3\dd x
\end{eqnarray*}

Since ${\bf c}[\mathcal{F}_1](p_0,p_1,p_\infty)={\bf c}[\mathcal{F}_2](p_0,p_1,p_\infty)$, they
share the same second factor, $Q_1=Q_2$. Denote $r=R(1,t)$ and $q=Q_1(1,t)$.
Since the map $\textup{hol}_{r,q}$ is a bijection by Proposition \ref{p:bijection}, the equality $H(\mathcal{F}_1)=H(\mathcal{F}_2)$ ensures that, in the normal forms, the terms $W_1$ and $W_2$ are also equal. Therefore, the foliations have the same normal forms and are formally equivalent. Now, following \cite{C}, they are also analytically equivalent.
\end{pfof}

Coming back to the contents of Proposition \ref{p:bijection}, we do not know whether the preimage of a holomorphic $k$-uple $(h_1,\ldots,h_k)$ by the map $\textup{hol}_{r,c}$ is convergent in full generality. However in some particular cases it is true:
\begin{prop}\label{p:convergence of hol}
  Let $n\geq 1$, $r(t)=t^n$ and $c(t)=t-1$. The restriction of $\textup{hol}_{r,c}$ to the space $\Omega_{r,c}$ of convergent elements in $\widehat{\Omega}_{r,c}$ defines a bijection onto the space $I^{0}_{n+1}$of convergent germs in $\widehat{I}^{0}_{n+1}$.
\end{prop}
\begin{pf}
  Let $h\in I^{0}_{n+1}$ be given. First consider $\omega_0=t^n\dd t+c(t)\dd x\in\Omega_{r,c}$. Since all germs at $0$ of order $n+1$ are locally conjugated there exists a local diffeomeorphism $\varphi\in\text{Diff}(\mathbb{C},0)$ such that $h=\varphi^{-1}\circ\text{hol}(\omega_0)\circ\varphi$. Next remark that $\omega_0$ can be extended to a foliation around the $(-1)$-curve $E_1=\{x=0\}$. On the chart $(y,u)$ defined by $t=1/u$ and $x=yu$ the foliation is defined by
  \begin{equation}\label{superform}
  \eta_0=(-1+uy\tilde{c}(u))du+(u^2\tilde{c}(u))\dd y
  \end{equation}
   where $\tilde{c}(u)=u^{n+1}c(1/u)$ is a polynomial.
  This foliation is in its turn defined around the curve $E_2=\{u=0\}$ and it can also be extended to a foliation around a $(-1)$-curve containing $u=0$. Indeed, it suffices to remark that $\eta_0$ is the blow up at the regular point of $-du+u\tilde{c}(u)dv$ by the map $v=yu$. On the other hand the radial foliation around $E_1$ defined by $dt$ and $du$ on respective charts extends also to the neighbourhood of the  $(-1)$-curve $E_2$. In this way we obtain a pair $(\mathcal{F}_0,\mathcal{G}_0)$ defined around $E_1\cup E_2$. The tangency locus between both foliations is invariant by $\mathcal{F}_0$ (and also by $\mathcal{G}_0$).

Remark that the change of coordinates between $\omega_0$ and $\eta_0$ can be described geometrically as follows. It is the only equivalence that acts like $(0,t)\mapsto(1/t,0)=(u,y) $ on a small annulus $\{0<|t|<r\}$ and sends the foliation pair $(\omega_0,dt)$ to the foliation pair $(\eta_0,du)$. If we choose a different gluing on the divisor, say $(0,t)\mapsto (1/\varphi(t),0)$, we can still extend to an equivalence defined on  a neighbourhood of the annulus by imposing that the pair  $(\omega_0,dt)$ is sent to the pair $(\eta_0,du)$. In this way we obtain a pair of foliations $(\mathcal{F},\mathcal{G})$ around a union of two $(-1)$-curves that we still call $E_1$ and $E_2$ satisfying that $\text{Tang}(\mathcal{F},\mathcal{G})$ is invariant by $\mathcal{F}$ and intersects $E_1$ transversally at two different points $q_1=E_2\cap E_1$ and $q_2\in E_1$. The divisor $\text{Tang}(\mathcal{F},E_1)$ has a unique point $p$ in its support. By construction, in the unique coordinate $w:E_1\rightarrow\mathbb{C}P^1$ for which $w(q_1)=0$, $w(q_2)=1$ and $w(p)=\infty$ two points $w_1,w_2\in \{w\in E_1:|w|>>1\}$ belong to the same leaf of $\mathcal{F}$ if and only if they belong to the same orbit of $1/h(1/w)$. In other words, the holonomy  $H(\mathcal{F})$ measured in the coordinate $s=1/w$ is precisely the series $h(s)$. Using the lemma \ref{l:U-1-1}, we get a pair of foliations $(\mathcal{F}_1,\mathcal{G}_1)\simeq (\mathcal{F},\mathcal{G})$ defined on a neighborhood of $\{0\}\times \mathbb{CP}^1$ in $\mathbb{D}\times \mathbb{CP}^1$. Let $(z,s)$ be some coordinates where the foliation $\mathcal{G}_1$ is defined by $ds$. In these coordinates $\mathcal{F}_1$ is defined by a holomorphic $1$-form $u(z)(s^n+b_{n-1}(z)s^{n-1}+\cdots+b_0(z))ds+(s-1)dz$ where $u$ is a unit. After changing the $z$-coordinate by $x\in(\mathbb{C},0)$ satisfying $\dd x=\frac{dz}{u(z)}$, the foliation $\mathcal{F}_1$ is defined by a holomorphic $1$-form \[\omega_1=(s^n+a_{n-1}(x)s^{n-1}+\cdots+a_0(x))ds+(s-1)\dd x\] for some $a_0,\ldots,a_{n-1}\in x\mathbb{C}\{x\}$. Blowing-up the point $(0,\infty)\in\mathbb{D}\times\mathbb{CP}^1$ does not change the expression of the normal form $\omega_1$ whose holonomy is by construction the holonomy of $\mathcal{F}$, that is to say $h$.
\end{pf}

Finally, we are now able to prove the theorem \ref{thm:invariantsinD}.

\begin{pfof}{Theorem \ref{thm:invariantsinD}}
To prove that $\mathfrak{I}$ is onto, let $D=r_1p_1+\ldots+r_kp_k\in\text{div} E(n)$ and $([D_1],(h_1,\ldots,h_k))\in{\bf E}_D$ be given. Consider an affine coordinate $s\in\mathbb{C}$ in $E$ containing all the points in the support of $D$ and $D_1$. These divisors can be respectively represented by monic polynomials $r(s)$ and $q(s)$ and the $1$-form $r(s)ds+q(s)\dd x$ defines an element such that
\begin{eqnarray*}
T(\mathcal{F}_0)&=&D\\
\mbox{div}(\F_0)&=&[D_1]
\end{eqnarray*}
We are going to do a surgery on $\F_0$ which preserves the above relations but modify the holonomies as expected. Consider the pair $(\mathcal{F}_0,\mathcal{G})$ where $\mathcal{G}$ is the radial foliation defined by $ds=0$ in the given coordinates. On a small neighbourhood $U_i$ of $p_i\in E$, following the lemma \ref{ptitlemme}, there exists a local biholomorphism $\phi_i$ sending the pair $(\mathcal{F}_0,\mathcal{G})_{|U_i}$ to the germ at $(0,0)$ of the pair $(\omega_i,dt)$ where $\omega_i=t^{r_i}dt+(t-1)\dd x$. Define $\varphi_i=\phi_{i|E\cap U_i}$. To keep track of the divisors $D$ and $D_1$ and change the holonomy we do not touch the gluing on $E\cap U_i$ but we change the model foliation $\omega_i$ by an appropriate holomorphic model $\widetilde{\omega}_i\in \Omega_{r=t^{r_i},(t-1)}$ obtained from Proposition \ref{p:convergence of hol} that satisfies $$h_i(s)=\varphi^{-1}_i\circ\text{hol}_{r,t-1}(\widetilde{\omega}_i)\circ \varphi_i(s).$$
Now, there exists a unique extension of $\varphi_i$ to a saturated neighbourhood $V_i$ of an annulus $A_i\subset U_i\cap E$ that sends the pair $(\mathcal{F},\mathcal{G})_{|V_i}$ to the pair of germs of foliations at $(0,0)$ defined by $(\widetilde{\omega}_i,dt)$. After doing this at each point $p_i$, we obtain a pair of regular foliations $(\mathcal{F},\mathcal{G})$ in a neighbourhood of a $(-1)$-curve. The contraction of this curve produces a foliation $\mathcal{F}\in\mathcal{D}^R$, $\mathcal{G}\in\mathcal{D}(0)$, $T(\mathcal{F})=D$, $\textup{div}(\mathcal{F})=[D_1]$ and $H(\mathcal{F})=(h_1,\ldots,h_k)$. Thus  the map $\mathfrak{I}$ is onto.

It remains to prove that for each $D=r_1p_1+\cdots+r_kp_k\in\text{div}(E)(n)$ and $${\bf e}=([D_1],(h_1,\ldots,h_k))\in{\bf E}_D $$ the fibre $\mathfrak{I}^{-1}({\bf e})$ is biholomorphic to $\mathbb{C}^M$ where $M$ is as in the statement of the theorem. Take coordinates $(x,y)$ of $(\mathbb{C}^2,0)$ such that $x=0$ does not define a point in the support of $D$. By Theorem \ref{t:formalnormalforms} we can assign to each class in $\mathfrak{I}^{-1}({\bf e})$ a unique formal $1$-form $W(x,y)(x\dd y-y\dd x)+Q(x,y)\dd x$
The homogeneous part of degree $n+2$ of $Q$ depends only on $[D_1]$ and thus is invariant in $\mathfrak{I}^{-1}({\bf e})$. Now, the polynomial $Q-Q_{n+2}$ is written
\[\sum_{\tiny \begin{array}{c}0\leq i\leq n-2 \\0\leq j\leq n-1 \\ i+j\geq n\end{array}}c_{ij}x^{i+3}y^j\]
and thus defines a point $(c_{ij})$ in $\mathbb{C}^M$. This defines a map from $\mathfrak{I}^{-1}({\bf e})$ to $\mathbb{C}^M$. Proposition \ref{p:bijection} ensures that this map is injective.

Let us prove that it is also onto.

Let $(c_{ij})$ in $\mathbb{C}^M$ be given. By construction, the family of complex numbers $(c_{ij})$, the class of divisor $[D_1]$ and the $k$-uple of holonomies  $h=(h_1,\ldots,h_k)$ is,associated to a formal normal form of the theorem \ref{t:formalnormalforms}. However, we are looking for a convergent foliation. The idea is basically the same as in the proposition \ref{p:convergence of hol}. We are going to do a \emph{very tangent} surgery on a polynomial form that is very tangent to the formal normal form. Indeed, using the construction of theorem \ref{t:formalnormalforms} it is easy to see that for any $N>>1$ there exists a foliation $\mathcal{F}_0$ defined by a \emph{polynomial} normal form $\omega_0=W_0(x,y)(x\dd y-y\dd x)+Q(x,y)\dd x$ satisfying that $$j^{N}(H(\mathcal{F}_{0}))=j^N(h).$$
After the blow-up $y=tx$, we consider the radial foliation $\mathcal{G}_0$ defined by $dt=0$. On a neighbourhood $U_i$ of $p_i\in|D|$, there exists a unique biholomorphism onto a neighbourhood of $(0,0)$ in coordinates $(z,t)$ sending $(0,t)$ to $(0,t)$ and the pair $(\mathcal{F},\mathcal{G})_{|U_i}$ to the pair $(\omega^0_i,dt)$ where
$\omega^0_i=(t^{r_i}+a^0_{r_i-1}(z)t^{r_i-1}+\cdots+a^0_{0}(z))dt+(t-1)dz$ is the convergent $1$-form of Proposition \ref{p:convergence of hol} having the same holonomy $h^0_i$ as $\mathcal{F}_0$ at $p_i$. The convergent $1$-form $\omega_i=(t^{r_i}+a_{r_i-1}(z)t^{r_i-1}+\cdots+a_{0}(z))dt+(t-1)dz$
for which $H(\mathcal{F}_{\omega_i})=h_i$ satisfies that $$\omega_i-\omega^0_i=z^{N'}f(z,t)dt$$ for some big $N'\in\mathbb{N}$. In particular this means that for each $0<\varepsilon_i<|t_0|<r_i$ for some small $\varepsilon_i$ and $r_i$, the holonomy germs $g^0_i,g_i:(\{t=t_0\},0)\rightarrow (E,t_0)$ of $\omega^0_i$ and $\omega_i$ respectively are very tangent. In particular, $g_i^{-1}\circ g^0_i(z)-z$ has a zero of order $N''$ that is as big as necessary. The equivalence between $\mathcal{F}_{\omega^0_i}$ and $\mathcal{F}_{\omega_i}$ in a neighbourhood of the annulus can be written as \begin{equation}(z,t)\mapsto(z+z^{N^{(3)}}(\cdots),t).\label{eq:cocycleorderN'}\end{equation}
for a $N^{(3)}$ as big as necessary. We claim that the gluing of $\mathcal{F}_{0|(\widetilde{\mathbb{C}^2},E)\setminus U_1\cup\cdots\cup U_k}$ with the models $\mathcal{F}_{\omega_i}$ via the given gluing (\ref{eq:cocycleorderN'}) defines an element in $\mathcal{F}\in \mathfrak{I}^{-1}({\bf e})$ with the desired invariants. Indeed, after the gluing, we get a pair $(\mathcal{F},\mathcal{G})$ of foliations around a rational curve of self-intersection $-1$. The proof of the classical version of the Grauert's theorem on the rigidity of neighborhood of rational $-1$-curve can be refined to ensure the following property: let $\textup{Diff}_{N}(E)$ be the sheaf over $E$ of germs of automorphisms that are locally written
$(z,t)\mapsto (z+z^N(\cdots),t)$ where $z=0$ is a local equation of $E$ and $\dd t=0$ is the radial foliation $\mathcal{G}$. Then,
  \[H^1\left(E,\textup{Diff}_{N}(E)\right)=0.\]
  This implies in particular that one can choose a holomorphic 1-form $\omega$ representing $\mathcal{F}$ as tangent as necessary to $\omega_0$. In particular, one can choose $N$ such that $\omega$ and $\omega_0$ coincides up to degree $2n$. Therefore, $\F$ has the desired invariants $(c_{ij})\in\mathbb{C}^M$ and the map from $\mathfrak{I}^{-1}({\bf e})$ to $\mathbb{C}^M$ is onto.
\end{pfof}

\subsection{Unfoldings versus deformations in $\mathcal{D}$.}
\label{s:deformation vs unfolding}
In this section, we give some independent results which compare the unfoldings and the deformations of a foliation in $\mathcal{D}$. In general, it is very difficult to give a criterion that recognizes deformations that underlie an unfolding. Indeed, a deformation of some one form $\omega_0$ in two variables $(x,y)$ with one parameter $t$ is written
\[\omega_t=\omega_0+t\left(a(x,y,t)\dd x+b(x,y,t)\dd y\right).\]
It underlies an unfolding if and only if there exists $c(x,y,t)$ such that
\[\omega_t+c(x,y,t)\dd t\] is integrable. The associated partial differential equation on $c$ is hardly solvable. In general, apart from this definition, no simple criteria are known to decide whether a deformation forms an unfolding or not. However, for the class $\mathcal{D}$, it can be read on the variation of the holonomy.

\begin{prop}
\label{p:deformationvsunfolding}
  Let $\{\mathcal{F}_c\}_{c\in (\mathbb{C}^k,0)}\subset\mathcal{D}$ be an analytic germ of deformation of $\mathcal{F}_0$ satisfying that $H(\mathcal{F}_c)=H(\mathcal{F}_0)$ for all $c\in(\mathbb{C}^k,0)$. Then the given deformation underlies  an equisingular unfolding of $\mathcal{F}_0$ on $(\mathbb{C}^{2+k},0)$.
\end{prop}

\begin{pfof}{Proposition \ref{p:deformationvsunfolding}}
  The idea is to find open sets where the deformation is a product, and hence can be thought of as an unfolding. In fact, to define an unfolding we need to find two open sets $U_1$ and $U_2$ in $\widetilde{\mathbb{C}^2}$ satisfying that $U_1\cup U_2$ is a neighbourhood of $E$, and on each $U_j$ an analytic family $\psi^j_c:U_j\rightarrow U_j$ of biholomorphisms fixing every point of $E$ and sending $\mathcal{F}_{0|U_j}$ to $\mathcal{F}_{c|U_j}$.
  In the present case the construction of $U_1$, $U_2$ and $\psi^j_c$ will be carried with the help of radial foliations.

  \begin{lem}\label{l:radials4unfoldings}
    Given $\mathcal{F}\in\mathcal{D}$ there exist two radial foliations $\mathcal{G}_1, \mathcal{G}_2\in\mathcal{D}(0)$ such that $$|\textup{Tang}(\mathcal{F},\mathcal{G}_1)_{|E}|\cap |\textup{Tang}(\mathcal{F},\mathcal{G}_2)_{|E}| =\emptyset$$
  \end{lem}
\begin{proof} Up to some change of coordinates, we can suppose that
\[\omega(x,y)=R(x,y)(x\dd y-y\dd x)+P_{n+2}\dd x+Q_{n+2}\dd y+\ldots\] representing $\mathcal{F}$ satisfies that the degree of $P_{n+2}(1,t)+tQ_{n+2}(1,t)$ is $n+3$. Since, $R(1,t)$ and $P_{n+2}(1,t)+tQ_{n+2}(1,t)$ has no common roots, one can choose $\mathcal{G}_1=x\partial_x +y\partial_y$ and $\mathcal{G}_2=x\partial_x +(y+x^2)\partial_y.$
\end{proof}
Applying Lemma \ref{l:radials4unfoldings} to $\mathcal{F}_0$, find a Jordan curve $\gamma$ in $E$ that separates  $|\text{Tang}(\mathcal{F}_0,\mathcal{G}_1)_{|E}|$ from $ |\text{Tang}(\mathcal{F}_0,\mathcal{G}_2)_{|E}|$ and take $U_1$ a $\mathcal{F}_0$ saturated neighbourhood of the disc in $E\setminus\gamma $ with no point in $|\text{Tang}(\mathcal{F}_0,\mathcal{G}_1)_{|E}|$, and $U_2$ an $\mathcal{F}_0$ saturated neighbourhood of the complementary disc. For sufficiently small $c\in(\mathbb{C}^k,0)$, $\mathcal{F}_c$ and $\mathcal{G}_j$ are transverse on $U_j$. We can define the equivalence $\psi^j_c$ on $U_j$ by simply imposing $\psi^j_{c|E}=\text{Id}$,  $(\psi^j_c)_*(\mathcal{F}_0)=\mathcal{F}_c$ and $(\psi^j_c)_*(\mathcal{G}_j)=\mathcal{G}_j$. The analytic dependence of $\mathcal{F}_c$ along the parameter guarantees that the families $\psi^j_c$ are analytic on $c\in(\mathbb{C}^k,0)$. To prove that the unfolding is equisingular we need to check that along the whole process of desingularization the leaves of the unfolding are transverse to the parameter fibration $(x,c)\mapsto c$. In the present case remark that a leaf of the unfolding is formed by the union of leaves of the deformation that sit over {\em the same}
point in $E$, since we do not move the points in $E$ along the construction. Thus the trace of a leaf of the unfolding on the divisor $E\times(\mathbb{C}^k,0)$ is $p\times(\mathbb{C}^k,0)$, which is regular and transverse to the fibration.
\end{pfof}
The previous argument allows to prove a global version (in the parameter space):
\begin{cor}
Let $U\subset\mathbb{C}^k$ be an open set and $\{\mathcal{F}_c\}_{c\in U}\subset\mathcal{D}$ be an analytic family satisfying that $c\mapsto(H(\mathcal{F}_c),\text{div}(\mathcal{F}_c))$ is constant on $U$. Then the given deformation underlies  an equisingular unfolding defined on $(\mathbb{C}^2,0)\times U$

\end{cor}

\begin{proof}
  The only difference with the previous case is that we can assume that $\text{Tang}(\mathcal{F}_c,\mathcal{G}_i)_{|E}$ does not depend on $c\in U$, and thus the open sets $U_1$ and $U_2$ can be fixed, and the equivalences constructed by the same method.
\end{proof}

\bigskip
\begin{minipage}{0.49\linewidth}
G. Calsamiglia\\
{\scriptsize Instituto de Matem\'{a}tica \\
Universidade Federal Fluminense\\
Rua M\' ario Santos Braga s/n\\
24020-140, Niter\'{o}i, Brazil \\
gabriel@mat.uff.br }
\end{minipage}
\begin{minipage}{0.49\linewidth}
Y. Genzmer\\
{\scriptsize I.M.T.\\
Universit\'e Paul Sabatier\\
Toulouse \\
yohann.genzmer@math.univ-toulouse.fr}
\end{minipage}

\end{document}